\documentclass[reqno,12pt]{amsart}
\usepackage{geometry}
\usepackage[numbers,sort&compress]{natbib}
\usepackage{graphicx}%
\usepackage{enumitem}
\usepackage{amsfonts,amsmath,amssymb,amsthm,mathtools}%
\usepackage[foot]{amsaddr}
\usepackage[title]{appendix}%
\usepackage{algorithm}
\usepackage{algorithmic}%
\usepackage{lmodern}
\usepackage[colorlinks=true,citecolor=blue]{hyperref}
\usepackage{caption}
\newcommand{\Lip}{\mathop{\rm Lip}}
\newcommand{\RR}{\mathbb R}  
\newcommand*{\para}[1]{\left({#1}\right)}
\def\NN{{\mathbb N}}      
\def\N{\mathcal{N}}      
\def\Fix{\mathop{\rm Fix}}      
\DeclareMathOperator{\diam}{diam}
\def\B{\mathcal{B}}
\def\F{\mathcal{S}}
\def\Ropt{R^*}
\def\Vopt{V_{\!\rho}}  
\def\bopt{B_{\!\rho}}
\def\r{r_{\!\rho}}
\def\Voptbis{\Vopt^{\hspace{0.1ex}\flat}}  
\def\boptbis{\bopt^{\hspace{0.1ex}\flat}}
\def\rbis{\r^{\hspace{0.1ex}\flat}}
\def\dbis{d^{\hspace{0.2ex}\flat}}
\def\Rbis{R^{\hspace{0.1ex}\flat}}
\def\betabis{\beta^{\hspace{0.1ex}\flat}}
\def\xflat{x^{\flat}_{\!\rho}}

\newcommand{\opthalpern}{{\sc m-opt-halpern}}
\newcommand{\flathalpern}{{\sc $\flat$-opt-halpern}}
\newcommand{\affhalpern}{{\sc aff-halpern}}
\newcommand{\adahalpern}{{\sc ada-halpern}}
\newcommand{\bpicard}{{\sc banach-picard}}

\newtheorem{theorem}{Theorem}
\newtheorem{proposition}[theorem]{Proposition}%
\newtheorem{example}{Example}%
\newtheorem{remark}{Remark}%
\newtheorem{definition}{Definition}%
\newtheorem{lemma}{Lemma}
\newtheorem{corollary}[theorem]{Corollary}

\title[]{Minimax-optimal Halpern iterations for Lipschitz maps
}

\author{Mario Bravo$^{1,\dagger}$}
\address{$^1$Departamento de Administraci\'on, Facultad de Administraci\'on y Econom\'ia,
Universidad de Santiago de Chile, Chile}
\email{mario.bravo.g@usach.cl}
\thanks{$^\dagger$All three authors contributed equally and are listed in alphabetical order.}
\author{Roberto Cominetti$^{2,\dagger}$}
\address{$^2$Institute for Mathematical and Computational Engineering and
Department of Industrial and Systems Engineering,
Pontificia Universidad Cat\'olica de Chile}
\email{roberto.cominetti@uc.cl}
\author{Jongmin Lee$^{3,\dagger}$}
\address{$^3$Department of Mathematical Sciences,
Seoul National University}
\email{dlwhd2000@snu.ac.kr}

\begin{document}

\begin{abstract}
This paper investigates the minimax-optimality of Halpern fixed-point iterations for Lipschitz maps in general normed spaces. Starting from an {\em a priori} bound on the orbit of iterates, we derive non-asymptotic estimates for the fixed-point residuals. These bounds are tight, meaning that they are attained by a suitable Lipschitz map and an associated  Halpern sequence. By minimizing these tight bounds we identify the minimax-optimal Halpern scheme. For contractions, the optimal iteration exhibits a transition from an initial Halpern phase to the classical Banach--Picard iteration and, as the Lipschitz constant approaches one, we recover the known convergence rate for nonexpansive maps. For expansive maps, the algorithm is purely Halpern with no Banach--Picard phase; moreover, on bounded domains, the residual estimates converge to the minimal displacement bound. Inspired by the minimax-optimal iteration, we design an adaptive scheme whose residuals are uniformly smaller than the minimax-optimal bounds, and can be significantly sharper in practice. Finally, we extend the analysis by introducing alternative bounds based on the distance to a fixed point, which allow us to handle mappings on unbounded domains; including the case of affine maps for which we also identify the minimax-optimal iteration.
\end{abstract}

\maketitle
\section{Introduction}

This paper investigates Halpern's fixed-point iteration for the class $\Lip(\rho)$ of all maps $T:C\to C$ defined on nonempty convex domains $C$ (not necessarily bounded nor closed) in arbitrary normed spaces 
$(X,\|\cdot\|)$ (not necessarily complete), and satisfying a Lipschitz condition with constant $\rho>0$ (possibly larger than 1), that is 
\begin{equation*}
(\forall x,y\in C)\quad \|Tx-Ty\| \le \rho\|x-y\|.
\end{equation*}
Given a sequence $\beta_n\in (0,1]$ and an initial point $x^0\in C$, Halpern's iterates \cite{h1967} are defined recursively  through the sequential averaging process 
\begin{equation} \label{eq:halp}\tag{$\textsc{h}$}
 (\forall\,n\geq 1)\quad x^n=(1\!-\!\beta_n)\,x^0+\beta_n\, Tx^{n-1}.
\end{equation}
Our goal is to establish tight upper bounds for the fixed point residuals $\|x^n\!-Tx^n\|$, and to use these bounds in order to determine the $\beta_n$'s that yield the best possible  iteration for contractions ($\rho<1$) and expansive maps ($\rho>1$),
extending the results obtained in \cite{cc2023} for the nonexpansive case ($\rho=1$). 

\subsection{Previous related work} For contractions the method of choice has been traditionally the Banach-Picard iteration {\sc (bp)} $x^{n}=Tx^{n-1}$, which guarantees
a geometric decrease of the residuals $\|x^n\!-Tx^n\|\leq\rho\|x^{n-1}\!-Tx^{n-1}\|$. In complete spaces, this implies that the iterates 
converge to the unique fixed point $x^*\!=Tx^*\!$. Moreover, setting $r_0=\|x^0\!-Tx^0\|$ and given some {\em a priori} estimate\footnote{Since $\|x^0\!-x^*\|\leq\|x^0\!-Tx^0\|+\|Tx^0\!-Tx^*\|\leq r_0+\rho\|x^0\!-x^*\|$ one can 
take $\delta_0=r_0/(1-\rho)$.}  $\delta_0\geq \|x^0\!-x^*\|$, this also yields the  
 error bounds
\begin{align}
\|x^n\!-Tx^n\|&\leq r_0\,\rho^n\leq \delta_0 \,(1\!+\!\rho)\rho^n,\label{eq:EBresidual}\\
\|x^n\!-x^*\|&\leq\delta_0\,\rho^n.\nonumber
\end{align}

Although these bounds tend to 0, the convergence rate deteriorates as $\rho$ increases, and at $\rho=1$ they become constant and uninformative with respect to convergence. In fact, nonexpansive maps may have no fixed points, and even when they exist, the {\sc (bp)} iterates  may fail to converge. For this case, several alternative methods have been proposed, including the  Krasnoselskii--Mann iteration \cite{krasnosel1955two,mann1953mean}, the Ishikawa iteration \cite{ishikawa1976fixed} and Halpern's iteration \cite{h1967,xu2002iterative}. Among these, 
Halpern stands out by achieving the best rate $\|x^n\!-Tx^n\|\sim O(1/n)$ in terms of fixed point residuals  (see \cite{bcc2022,cc2023,lie2021,ss2017}). 
The rationale is that choosing $\beta_n<1$ in \eqref{eq:halp} introduces some contraction in the map $T_nx=(1\!-\!\beta_n)\,x^0+\beta_n Tx$, so that letting $\beta_n\uparrow 1$ in a controlled manner the cumulative effect drives the 
residual norms to zero. Moreover, under suitable conditions on the norm, the iterates converge to a fixed point. For a survey on the convergence  of
Halpern's iterates we refer to \cite[L\'opez {\em et al.}]{lmx2010}. Instead, the focus of the present paper is  on explicit and optimal error bounds, not only for nonexpansive maps but also for contractions and general Lipschitz maps.

It turns out that Halpern iteration is also effective for contractions. Notably, in a Hilbert space setting,
\cite[Park and Ryu]{pr2022} showed that  $\beta_n\!=\!(1\!-\!\rho^{2n})/(1\!-\!\rho^{2n+2})$ yields a bound $\|x^n\!-Tx^n\|\leq \mbox{\sc pr}_n\!=\delta_0\,\rho^n(1\!-\!\rho^2)/(1\!-\!\rho^{n+1})$ that is smaller than  \eqref{eq:EBresidual}.
Moreover, \cite{pr2022} proved that 
this is the best one can achieve in Hilbert spaces, not just with Halpern but with any deterministic algorithm  whose iterates depend only on the previous residuals.
Interestingly, when $\rho\uparrow 1$ these $\beta_n$'s converge to $n/(n+1)$ and $\mbox{\sc pr}_n$ tends to $2\delta_0/(n+1)$, connecting smoothly with the tight bounds for nonexpansive maps in Hilbert spaces established in \cite[Lieder]{lie2021}.

In general normed spaces, Halpern iteration for nonexpansive maps behaves differently, nevertheless one can still derive error bounds of the form $\|x^n\!-Tx^n\|\leq 2\delta_0\,R_n$. \cite[Sabach and Shtern]{ss2017} proved that for $\beta_n\!=n/(n\!+\!2)$ a valid bound is $ R_n\!=4/(n+1)$. More generally, 
\cite[Bravo {\em et al.}]{bcc2022} showed that for every non-decreasing sequence $\beta_n$ this  holds with
$R_n\!=\!\sum_{i=0}^n(1\!-\!\beta_i)^2\prod_{k=i+1}^n\beta_k$,
and that this is tight in a minimax sense: there is a nonexpansive map and a corresponding Halpern sequence that attains these bounds with equality for all $n\in\NN$. For $\beta_n\!=n/(n\!+\!2)$ this yields the slightly improved tight bound
 $R_n\!=\!4(1\!-\!\frac{H_{n+2}}{n+2})/(n\!+\!1)$ with $H_n\!\!=\!\sum_{k=1}^n\!\frac{1}{k}$ the $n$-th harmonic number.
Similarly, the choice $\beta_n\!=n/(n\!+\!1)$ --- which was optimal for Hilbert spaces --- in normed spaces yields the tight bound $R_n\!=\!H_n/(n\!+\!1)$.  
Beyond these particular sequences $\beta_n$,  \cite[Contreras and Cominetti]{cc2023} proved that the optimal choice
is given by the recursive sequence $\beta_{n+1}\!=(1+\beta_{n}^2)/2$ with $\beta_0=0$, achieving a tight 
optimal bound $R_n$ that satisfies the recursion $R_{n+1}\!=R_{n}-\frac{1}{4}R_n^2$,  and behaves asymptotically 
as $R_n\!\sim 4/(n\!+\!1)$.  This highlights a genuine gap for Halpern iterations between Hilbert spaces and general normed spaces, stemming from their differing geometries. Furthermore, this optimal Halpern attains the smallest error among all Mann iterations, up to a multiplicative factor very close to 1 (see \cite{cc2023}).
We observe that the coefficient $2\delta_0$ in these bounds can be replaced by an {\em a priori} estimate $\kappa_0$ for the orbit of iterates, namely\footnote{When $T$ has a fixed point $x^*\!$, the iterates remain in the ball $B(x^*\!,\|x^0\!-x^*\|)$ and one can take $\kappa_0=2\delta_0$. However, a suitable $\kappa_0$ may be available without assuming a fixed point, such as when $T$ has bounded range.} $\|x^0\!-Tx^n\|\leq\kappa_0$ for all $n\in\NN$.

The ability to work in general normed spaces allows to deal in particular with the Bellman operator in $\rho$-discounted Markov decision processes, which is $\rho$-contractive in the infinity norm $\|\cdot\|_\infty$. For this map, 
\cite[Lee and Ryu]{lr2023} observed that 
introducing a parameter $\beta_n$ yields an improved error bound which does not degenerate as $\rho\uparrow 1$, and recovers in the limit the known sublinear error bounds for Halpern's iterates. In this setting, considering $\rho\approx 1$ is  relevant as it is frequently used to approximate average reward MDPs. In particular, \cite[Zurek and Chen]{Z2025} proposed an algorithm that deploys Halpern iteration for approximately $1/(1\!-\!\rho)$
steps, and then switches to a Banach-Picard iteration. This aligns with our results in Section~\ref{SS2}, where we show that the optimal Halpern iteration for contractions in normed spaces exhibits a similar feature, in sharp contrast with the Hilbert case.

Turning to the case of expansive maps with $\rho>1$, the results are scarce. In this case the existence of fixed points is not guaranteed and the 
fixed point residuals may remain bounded away from 0. In fact, \cite[Goebel]{g1973} proved that when $C$ is a nonempty and {bounded} convex domain in a Banach space, every $\rho$-Lipschitz map $T: C \to C$ with $\rho\geq 1$ satisfies the {\em minimal displacement} bound\footnote{Here we express the bound using the diameter of $C$ instead of the more general bound based on the Chebyshev radius as in \cite{g1973}.}
\begin{equation}\label{eq:mdb}
\inf_{x \in C }\|x -Tx\|\leq \diam(C)\big (1-1/\rho\big ),
\end{equation}
and that there are maps satisfying this with equality. Baillon and Bruck~\cite{bb1996} conjectured that the Krasnoselskii-Mann iterates would attain this bound
asymptotically, although this has not been confirmed.
Recently, \cite[Diakonikolas]{d2025} proposed  a Halpern iteration that attains the weaker bound $\limsup_{n\to\infty}\|x^n\!-Tx^n\|\leq \diam(C)\left (\rho-1\right )$. Here we improve on this, by showing  that
the optimal Halpern iteration  attains asymptotically  the exact minimal displacement bound \eqref{eq:mdb} for every $
\rho\geq 1$.

\subsection{Our contribution}
The fundamental question addressed in this paper is to determine the  parameters $\beta_n$ in Halpern iteration that yield the best possible error bounds for $\rho$-Lipschitz maps in general normed spaces. Our approach is  based  on minimizing suitable recursive bounds for the fixed-point residuals, adapting  techniques previously developed for the analysis of general averaged iterations for nonexpansive maps \cite{bcc2022,bc2018,cc2023}. 
We consider several variants that differ in how the parameters are selected, depending on the structural assumptions imposed on the operator.
A summary of our contributions is as follows.

\vspace{2ex}
\noindent $\bullet$ Given an arbitrary sequence $(\beta_n)_{n\in\NN}\subseteq[0,1]$, and assuming some {\em a priori} estimate $\kappa$ for the orbit of iterates $(x^n)_{n\in\NN}$, we obtain  non-asymptotic bounds $\|x^n \!- T x^n\|\leq\kappa\,R_n$ for the fixed-point residuals, which hold for every $T\in\Lip(\rho)$ (Proposition~\ref{Prop_2.1}). These bounds are shown to be tight, in the sense that there exists a $\rho$-Lipschitz map and a corresponding Halpern sequence for which they hold with equality for all $n \in \NN$. We emphasize that the bounds $R_n$ are dimension-free and depend solely on $\rho$ and the $\beta_n$'s.

\vspace{2ex}
\noindent $\bullet$ Theorem~\ref{Thm:Main} then shows how a minimax-optimal iteration,  named \opthalpern, can be derived by recursively minimizing the bounds $R_n$ with respect to the $\beta_n$'s.

\begin{itemize}[leftmargin=0.6cm,itemsep=0.2cm]
    \item[-]
For $\rho < 1$, the optimal iteration exhibits an initial phase of Halpern iterates with $\beta_n<1$, followed by a standard Banach--Picard iteration with $\beta_n\equiv 1$. For $\rho \uparrow 1$ we recover the known tight bound $O(1/n)$  for nonexpansive maps, which suggests that in regimes where $\rho \approx 1$ this optimal algorithm could outperform Banach--Picard. This is formally established in \S\ref{secbp} and is illustrated with a simple numerical example in \S\ref{sec:numil}. Furthermore, \S\ref{secpr} compares the minimax-optimal Halpern iteration in
normed spaces to the optimal iteration in Hilbert settings, showing that the gap between the two is at most a factor $e^2\approx 7.39$ (see Figure \ref{figura2}).

\item[-] For $\rho \geq 1$,  the optimal iteration is purely Halpern and never switches to Banach-Picard. Moreover, for bounded domains, we show that the residual bounds converge to the minimal displacement bound \eqref{eq:mdb}. We are not aware of such a result in the literature of fixed-point iterations.
\end{itemize}

\vspace{1ex}
\noindent $\bullet$ 
In contrast with the minimax-optimal iteration, which depends solely on $\rho$ and is otherwise agnostic to the specific map $T$, section \S\ref{sec2.4} introduces an adaptive variant, named \adahalpern, in which the parameters are selected online based on observed iterates, using a modified sequence of recursive bounds that are uniformly smaller than those of the minimax-optimal iteration (Theorem~\ref{Thm:ada_halpern}). A simple numerical example demonstrates that this adaptive scheme can be significantly faster, particularly when $T$ has a contraction  constant smaller than $\rho$ (possibly for a different norm).

\vspace{2ex}
\noindent $\bullet$ The last Section \S\ref{sec3} considers two alternative recursive bounds based directly on  an estimate $\delta_0\ge \|x^0\!-x^*\|$ for the distance between $x^0$ and a fixed point $x^*\!$. The first bound is the basis for the \flathalpern{} iteration, which can deal with nonlinear maps on possibly unbounded
domains (see Theorem~\ref{Thm:unbounded} and the subsequent discussion). The second one is specific for affine maps and leads to the minimax-optimal iteration \affhalpern{} (Theorem~\ref{Thm:affine}).
We close the paper with a numerical comparison of these
alternative iterations, applied to a simple linear map.

\vspace{-1ex}
\section{Halpern iterations with  bounded orbits}\label{sec2}

 Let $T:C\to C$ be an arbitrary map in $\Lip(\rho)$ and consider an orbit $(x^n)_{n\in\NN}$ of Halpern iterates produced by \eqref{eq:halp}. 
Throughout this section we assume that the iterates remain bounded, or equivalently, that the images $Tx^n$ are bounded.
Specifically, we assume that there exist constants $\kappa_0$ and $\kappa_1$ with $\|x^0\!-Tx^n\|\leq \kappa_0$  and $\|Tx^m\!-Tx^n\|\leq \kappa_1$ for all $m,n\in\NN$. Denoting $\kappa\triangleq\max\{\kappa_0,\kappa_1\}$ and  $\mathcal N\triangleq\NN\cup\{-1\}$, and adopting the convention $Tx^{-1}=x^0$, we then have
\begin{equation}\label{eq:bkappa}\tag{$\textsc{b}_\kappa$}
 (\forall\,m,n\in \N)\quad\|Tx^m\!-Tx^n\|\leq\kappa.
\end{equation}

\begin{remark} If $C$ is a bounded domain one can take $\kappa=\mathop{\rm diam}(C)$. Similarly, if $T$ has a bounded range, then \eqref{eq:bkappa} holds with $\kappa=\|x^0\!-Tx^0\|+\diam(T(C))$. When $\rho\leq 1$ and there  is a fixed point $x^*\!=Tx^*\!$, one can show inductively that $\|x^n\!-x^*\|\leq\|x^0\!-x^*\|$ and
 \eqref{eq:bkappa} holds with $\kappa=(1\!+\!\rho)\|x^0\!-x^*\|$. In general, if we have a valid bound $\kappa_0$, by triangle inequality one can take $\kappa_1=2\kappa_0$. Conversely, if we have a valid $\kappa_1$ we can set $\kappa_0=\|x^0\!-Tx^0\|+\rho\,\kappa_1$. 
\end{remark}

The {\em a priori} estimate $\kappa$
 will be used to establish tight upper bounds of the form $\|x^n\!-Tx^n\|\leq\kappa\,R_n$ for the fixed point residuals. These bounds will be later minimized in order to determine the choice of $\beta_n$'s that yield the best Halpern iteration. Let us observe that $\kappa_1$ is only relevant when dealing with  expansive maps. In fact, as shown below, when $\rho\leq 1$ one can take $\kappa_1=\rho\,\kappa_0$, and hence $\kappa=\kappa_0$.

\begin{lemma}
    Let $\rho\leq 1$ and $\kappa_0\geq\sup_{n\in\NN} \|x^0\!-Tx^n\|$. Then $\|Tx^m\!-Tx^n\|\leq \rho\,\kappa_0$. 
\end{lemma}

\begin{proof} It suffices to show that $\|x^m\! - x^n\|\leq \kappa_0$. Given $n,m \in \NN$ and assuming without loss of generality that $\beta_n\geq\beta_m$, we have $x^n-x^m=(\beta_n -\beta_m)(Tx^{n-1} -x^0) + \beta_m (T x^{n-1} - T x^{m-1})$ so that $\|x^n-x^m\|\leq (\beta_n -\beta_m)\kappa_0 + \rho \beta_m \|x^{m-1} - x^{n-1}\|$. The result  then follows by performing a double induction on $n$ and $m$.  
\end{proof}
\subsection{Minimax optimality through recursive bounds}\label{SS2}
Let us set $\beta_0=0$, $d_{0}=0$, $c_0=0$, $R_0=1$, and define recursively
\begin{equation}\label{eq:rec}\tag{$\textsc{r}$}
(\forall\,n\geq 1)\quad\left\{\begin{array}{l}
    d_n\,=\;|\beta_{n-1}-\beta_n|+\min\{\beta_{n-1},\beta_n\}\,c_{n-1}\\
    c_n\;=\;\min\{1,\rho\, d_n\}\\
    R_n=\;1-\beta_n(1-c_n)
\end{array}\right.
\end{equation}

\begin{proposition}\label{Prop_2.1}
If \eqref{eq:bkappa} is satisfied then $\|x^n\!-x^{n-1}\|\leq\kappa\, d_n$ and $\|x^n\!-Tx^n\|\le \kappa\,R_n$ for all $n\in\NN$. These bounds are minimax tight: for every $\rho>0$, $\kappa>0$, and  $(\beta_n)_{n\in\NN}$, there exists $T\in\Lip(\rho)$ and a corresponding Halpern sequence satisfying \eqref{eq:bkappa},
with $\|x^n\!-x^{n-1}\|= \kappa\,d_n$ and $\|x^n\!-Tx^n\|=\kappa\,R_n$ for all $n\geq 1$.
\end{proposition}

\begin{proof}
By rescaling the norm, we may assume without loss of generality that $\kappa=1$.
For $n=1$ we clearly have $\|x^1\!-x^0\|=\beta_1\|x^0\!-Tx^0\|\leq\beta_1=d_1$. Inductively, for $n\geq 2$ we have $\|Tx^{n-1}\!-Tx^{n-2}\|\leq\min\{1,\rho\,d_{n-1}\}=c_{n-1}$ so that
\begin{align*}
\|x^n\!-x^{n-1}\|
&=\|(\beta_{n-1}-\beta_n)(x^0\!-Tx^{n-1})+\beta_{n-1}(Tx^{n-1}\!-Tx^{n-2})\|\\
&\leq |\beta_{n-1}-\beta_n|+\beta_{n-1}\,c_{n-1},\\[1ex]
\|x^n\!-x^{n-1}\|
&=\|(\beta_{n-1}-\beta_n)(x^0\!-Tx^{n-2})+\beta_{n}(Tx^{n-1}\!-Tx^{n-2})\|\\
&\leq|\beta_{n-1}-\beta_n|+\beta_{n}\,c_{n-1}.
\end{align*}
The minimum of these two bounds yields $\|x^n-x^{n-1}\|\leq d_n$, and from this we also get
  \begin{align*}
\|x^n\!-Tx^n\|&\leq(1-\beta_n)\|x^0\!-Tx^n\|+\beta_n\|Tx^n\!-Tx^{n-1}\|\\
&\leq (1-\beta_n)+\beta_n\,c_n=R_n.
\end{align*}
The tightness  is more technical and is proved as Corollary \ref{Cor_6.6Lip} in Appendix \ref{AppendixA}.
\end{proof}

\begin{remark}
{\em Appendix \ref{AppendixA}} derives tight error bounds for general Mann iterates where $x^n$ is a convex combination of $\{x^0, Tx^0, \ldots, Tx^{n-1}\}$, namely
$$\mbox{$
x^n= \sum_{i=0}^n \pi_i^n T x^{i-1}
$}$$ 
with $T x^{-1}\!=x^0$ and 
$\pi^n_i\ge 0$ satisfying $\sum^n_{i=1}\pi^n_i=1$ and $\pi^n_n>0$. Clearly Halpern's iteration is a special case of this general scheme. The main argument adapts ideas from the nonexpansive case {\rm \cite{bcc2022,bc2018,cc2023}} to the case where $\rho \neq 1$, using a nested family of optimal transport problems whose solutions provide tight bounds for the distance between iterates $\|x^n\!-x^m\|$ as well as for the  residuals $\|x^n\!-Tx^n\|$ (see {\rm Proposition~\ref{prop_6.2Lip}} and {\rm Theorem~\ref{thm:tightLip}}).
\end{remark}

Notice that $R_n$ depends only on the previous $\beta_k$'s, namely
$R_n=R_n(\beta_1,\ldots,\beta_n)$. Since the error bounds $\|x^n\!-Tx^n\|\leq\kappa\, R_n$ are tight, a minimax-optimal iteration can be obtained by solving for each $n\geq 1$
\begin{equation*}
R_n^*=\min_{(\beta_1,\ldots,\beta_n)\in[0,1]^n}R_n(\beta_1,\ldots,\beta_n).
\end{equation*}
Finding this global minimum $R_n^*$ and optimal coefficients $(\beta_1^*,\ldots,\beta_n^*)$ is challenging due to the recursive structure of $R_n(\cdot)$, which involves absolute values in the $d_k$'s and ``min'' operations in the $c_k$'s. However, we will show that the optimum can be obtained by minimizing sequentially with respect to each $\beta_n$, one at a time. The proof is elementary, albeit not entirely trivial, and rests on the following observation.
\begin{lemma} \label{Lema1} The minimum $R_n^*$ is attained
in the set $\F_n=\{\boldsymbol{\beta}\in[0,1]^n:\beta_n\geq\beta_{n-1}\}$. Moreover we have $R_n^*<1$ and 
\begin{align}
R_n^* &=\min_{\boldsymbol{\beta}\in \F_n} 1-\beta_n+\rho\,\beta_n^2+\rho\,\beta_n(R_{n-1}(\beta_1,\ldots,\beta_{n-1})-1).\label{OP}
    \end{align}   
\end{lemma}

\begin{proof}
Substituting $c_n\!=\min\{1,\rho\,d_n\}$  in the expression for $R_n$ yields
\begin{align}\label{eq:yy}
R_n&=\min\{1,1-\beta_n+\rho\,\beta_n\,d_n\}
\end{align}
with $d_n$ defined recursively by $d_n=|\beta_{n-1}-\beta_n|+\min\{\beta_{n-1},\beta_n\}\,c_{n-1}$.
Now, since $c_{n-1}\leq 1$, by choosing $\beta_n=\beta_{n-1}$ small and strictly positive  we can make $\rho\,d_n<1$ and then $R_n^*<1$.
This implies that  the minimum of $R_n$ coincides with the minimum of the inner expression in \eqref{eq:yy}, that is $\tilde R_n= 1-\beta_n+\rho\,\beta_n\,d_n$.

When $\beta_n\in[0,\beta_{n-1}]$ we have $d_n=\beta_{n-1}-\beta_n+\beta_n\,c_{n-1}$ which gives 
\begin{align*}
\tilde R_n&=1-\beta_n(1-\rho\beta_{n-1})-\rho\beta_n^2(1-c_{n-1}).
\end{align*}
This is concave quadratic in $\beta_n$ so its minimum is either at $\beta_n=0$ or $\beta_n=\beta_{n-1}$. However, $R_n^*<1$ excludes $\beta_n=0$ as a candidate for a 
minimum, and therefore the minimum of $R_n(\cdot)$ is attained with $\beta_n\geq\beta_{n-1}$. Now, for $\beta_n\geq\beta_{n-1}$ we have
 $d_n=\beta_n-\beta_{n-1}(1-c_{n-1})=\beta_n+(R_{n-1}-1)$, so that
 \begin{align*}
\tilde R_n &= 1-\beta_n+\rho\,\beta_n^2+\rho\,\beta_n(R_{n-1}-1)
    \end{align*}  
    from which we conclude \eqref{OP}.
    \end{proof}

Since a smaller term $R_{n-1}(\beta_1,\ldots,\beta_{n-1})$ in \eqref{OP} contributes to reducing $R_n^*$, this suggests a sequential strategy where we freeze $\beta_k=\beta_k^*$
for $k\le n-1$  at the optimal parameters for $R_{n-1}^*$, and then solve a one-dimensional quadratic minimization for $\beta_n\in[\beta_{n-1}^*,1]$, with $R_{n-1}(\beta_1,\ldots,\beta_{n-1})$ replaced by  its minimum $R_{n-1}^*$.
However, fixing $\beta_{n-1}^*$ also restricts the feasible interval for $\beta_n$, which might offset the gain obtained from the smaller factor $R_{n-1}^*$ in the objective function.

Thus, in order to justify this sequential scheme we must check that the optimal parameters $(\beta_1^*,\ldots,\beta_{n-1}^*)$  for $R_{n-1}^*$ remain optimal for
$R_n^*$. 
The proof uses basic facts about the following  quadratic problems parameterized by $r\in[0,1]$
$$\Vopt(r)=\min_{\beta\in [0,1]} 1-\beta+\rho\, \beta^2+\rho\,\beta\,(r-1).$$
The unconstrained minimizer is
 $\beta_\rho(r)=\frac{1}{2}(1/\rho+1-r)>0$,
 which decreases with $r$, and the minimum over $[0,1]$ is 
 attained at $\bopt(r)=\min\{1,\beta_\rho(r)\}$, so that
$$\Vopt(r)=\left\{\begin{array}{cl}
1-\rho\,\beta_\rho(r)^2&\mbox{if }r\geq 1/\rho-1\\
\rho \,r&\mbox{if }r\le 1/\rho-1
\end{array}\right.
$$
which is strictly increasing  and continuous for $r\in[0,1]$ (see Figure~\ref{figura1}). Notice that when $\rho\geq 1$ the first case applies  to all $r\in[0,1]$, and if $\rho\leq\frac{1}{2}$ the second case  holds for all $r\in[0,1]$.
A transition at $1/\rho-1\in (0,1)$ occurs only
when $\frac{1}{2}<\rho<1$. A final remark is that a direct verification (distinguishing the cases $\rho\le 1$ and $\rho\geq 1$)  shows that $\Vopt(r)<r$ except at $\r\triangleq\max\{0,1-1/\rho\}\in [0,1)$ where $\Vopt(\r)=\r$.  

\begin{figure}
  \centering
  \captionsetup{width=\linewidth}
  \begin{minipage}{.47\linewidth}
    \centering
    \includegraphics[width=\linewidth]{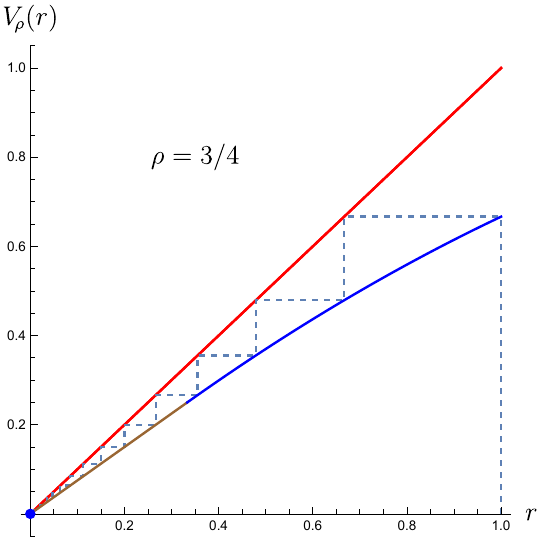}
  \end{minipage}\hfill
  \begin{minipage}{.47\linewidth}
    \centering
    \includegraphics[width=\linewidth]{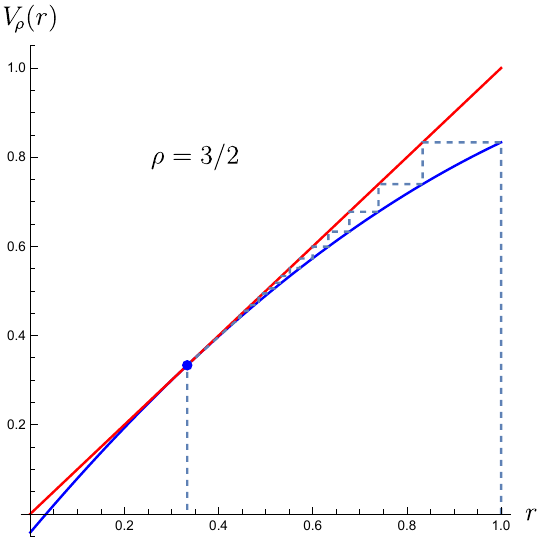}
  \end{minipage}
  \caption{\footnotesize
  Behavior of $r\mapsto \Vopt(r)$ for different values of $\rho$.
  The dashed lines illustrate the convergence of the iterates
  $r_{n}=\Vopt(r_{n-1})$ started from $r_0=1$, towards
  $\r\triangleq\max\{0,1-1/\rho\}\in[0,1)$, the unique solution of
  $r_\rho=\Vopt(r_\rho)$ (fat dot).
  The plot on the left is for $\rho=3/4$ with $\r=0$ and a linear
  regime for $r\leq 1/3$ (brown), whereas the right plot is for
  $\rho=3/2$ with $\r=1/3$ and no change of regime.
  In this latter case the curve $\Vopt(r)$ is tangent to the diagonal
  at $\r$.}
  \label{figura1}
\end{figure}

\begin{proposition}\label{Teo2}
The minimax-optimal Halpern iteration for the class $\Lip(\rho)$  is 
achieved with the sequence defined recursively by 
$\beta_{n}^*\!=\!\bopt(R_{n-1}^*)\!=\!\min\{1,\frac{1}{2\rho}(1+(\rho\beta_{n-1}^*)^2)\}$ with  $\beta_0^*\!=0$. The optimal bounds satisfy the recursion $R_n^*=\Vopt(R_{n-1}^*)$ with $R_0^*=1$, which yields a strictly decreasing sequence that converges to $\r=\max\{0,1-1/\rho\}\in [0,1)$.
\end{proposition}

\begin{proof}
Let us prove inductively that $R_n^*=\Vopt(R_{n-1}^*)<R_{n-1}^*$ with optimal $\beta_k^*=\bopt(R_{k-1}^*)$ for $k=1,\ldots,n$. 
The base case $n=1$ is clear since $R_0^*=1$ and $\beta_0=0$, so that the minimum $R_1^*$ in \eqref{OP} is precisely $R_1^*\!=\Vopt(R_0^*)<1=R_0^*$, attained at $\beta_1^*\!=\!\bopt(R_0^*)$. 

Suppose that the property holds for $n-1$, and let us prove it remains valid for $n$. 
By adding the constraint $\beta_n\geq\beta_{n-1}^*$ to the minimum that defines 
$\Vopt(R_{n-1}^*)$, and using equation \eqref{OP} in Lemma \ref{Lema1}, it follows directly that 
\begin{align*}
\Vopt(R_{n-1}^*)&\leq\min_{\beta_n\geq\beta_{n\!-\!1}^*}1-\beta_n+\rho\,\beta_n^2+\rho\,\beta_n(R_{n-1}^*-1)= R_n^*.
\end{align*}
For the reverse inequality, let $\beta_n^*=\bopt(R_{n-1}^*)$. 
Since $\beta_{n-1}^*\!=\bopt(R_{n-2}^*)$ and $\bopt(\cdot)$ is decreasing, the induction hypothesis $R_{n-1}^*\!\!<\!R_{n-2}^*$ implies $\beta_n^*\!\geq\!\beta_{n-1}^*$ so that $(\beta_1^*,\ldots,\beta_n^*)\in \F_n$.
Using  \eqref{OP} once more, and since $R_{n-1}^*\!=R_{n-1}(\beta_1^*,\ldots,\beta_{n-1}^*)$, we get
\begin{align*}
   \Vopt(R_{n-1}^*)&=1-\beta_n^*+\rho\,(\beta_n^*)^2+\rho\,\beta_n^*(R_{n-1}(\beta_1^*,
\ldots,\beta_{n-1}^*)-1)\geq R_n^*.
\end{align*} 
Finally, since $\Vopt(\cdot)$ is strictly decreasing, the induction hypothesis $R_{n-1}^*<R_{n-2}^*$ implies
 $R_n^*=\Vopt(R_{n-1}^*)<\Vopt(R_{n-2}^*)=R_{n-1}^*$ completing the induction step.
 
This shows that the recursive minimization $R_n^*=\Vopt(R_{n-1}^*)$ achieves the global minimum of $R_n(\beta_1,\ldots,\beta_n)$ over $[0,1]^n$.
Moreover, the sequence $R_n^*$ is strictly decreasing and converges to $r_\rho$ the unique solution of $\Vopt(r)=r$.
On the other hand, the optimal sequence $\beta_n^*$ increases, and  $\beta_{n}^*<1$ if and only if $R_{n-1}^*>1/\rho-1$, in which case 
$\beta_{n}^*=\beta_\rho(R_{n-1}^*)$ and $R_{n}^*=\Vopt(R_{n-1}^*)=1-\rho\,(\beta_n^*)^2$, so that $\beta_\rho(R_n^*)=\frac{1}{2}(1/\rho+1-R_n^*)=\frac{1}{2\rho}(1+(\rho\,\beta_n^*)^2)$. This proves that the optimal parameters satisfy the recursion
 $\beta_{n+1}^*=\min\{1,\frac{1}{2\rho}(1+(\rho\,\beta_n^*)^2)\}$.
\end{proof}
A pseudo-code for this minimax-optimal Halpern iteration is presented below.
For nonexpansive maps with $\rho=1$ it recovers the 
method  in \cite{cc2023} with $\beta_{n}^*=\frac{1}{2}(1+(\beta_{n-1}^*)^2)$ and $R_{n}^*=R_{n-1}^*-\frac{1}{4}(R_{n-1}^*)^2$. 

\begin{algorithm}[ht!]
  \caption{\mbox\opthalpern{} (for $T:C\to C$ a $\rho$-Lipschitz map)}
  \begin{algorithmic}
    \STATE select $x^0\in C$ and set $R_0^*=1$ 
    \FOR{ $n=1,2,\ldots$}
    \STATE compute   $\beta_n^*=\bopt(R_{n-1}^*)$ and set $R_n^*=\Vopt(R_{n-1}^*)$
      \STATE update $x^n=(1-\beta_n^*)x^0+\beta_n^*\,Tx^{n-1}$
    \ENDFOR
  \end{algorithmic}
\end{algorithm}

 Notice that the optimal sequences $\beta_n^*=\bopt(R_{n-1}^*)$ and $R_n^*=\Vopt(R_{n-1}^*)$  only depend on $\rho$ and not on the specific map $T$ nor on the {\em a priori} estimate $\kappa$. Furthermore, a close inspection of the proof of Proposition \ref{Teo2} reveals that for this optimal sequence one can disregard $\kappa_1$ and obtain error bounds that  involve  just $\kappa_0$. Below we state this result,  which extends \cite{cc2023} from the nonexpansive setting to 
 mappings with arbitrary Lipschitz constants $\rho$. We subsequently examine the nature of this extension, considering separately the cases $\rho<1$ and $\rho>1$.

\begin{theorem}\label{Thm:Main} Let $T\in\Lip(\rho)$ and consider the iterates $(x^n)_{n\in\NN}$ generated by \opthalpern{}. Suppose that there exists $\kappa_0$ such that $\|x^0\!-Tx^n\|\leq\kappa_0$ for all $n\in\NN$. Then, the error bound $\|x^n\!-Tx^n\|\leq\kappa_0\,R_n^*$ holds for all $n\in\NN$. This bound is tight: there exists a map  $T\in\Lip(\rho)$ and a corresponding Halpern sequence that satisfies all these bounds with equality.
\end{theorem}

\begin{proof} Repeating the arguments in Proposition \ref{Prop_2.1}, we obtain $\|x^n\!-Tx^n\|\leq\kappa_0\, R_n$ 
with $R_n$ is defined as in \eqref{eq:rec}, but replacing $c_n=\min\{1,\rho\, d_n\}$ with the larger bound $c_n=\rho\, d_n$.
Then, by following the proof of Proposition \ref{Teo2}, we notice that the optimal sequence $\beta_n^*$'s is the same, which yields
 $\|x^n\!-Tx^*\|\leq\kappa_0\,R_n^*$.
\end{proof}
 
 \subsection{The contractive case  (\texorpdfstring{$\rho<1$}{})}
 
 When $\rho<1$ the optimal iteration satisfies $R_n^*\downarrow r_\rho=0$ and exhibits an initial phase with $\beta_n^*<1$, after which it proceeds as a standard Banach-Picard with $\beta_n^*\equiv 1$ in all subsequent iterations. The transition occurs as soon as $R_n^*$ drops below the threshold $1/\rho-1$, and thereafter $R_n^*$ tends to 0 at a linear rate (see left panel in Figure \ref{figura1}). As a consequence, in complete spaces, the iterates $x^n$ converge to the unique fixed point $x^*$.
  
 When $\rho \in (0,1/2]$ this threshold exceeds 1 and there is no initial phase: in this regime, Halpern's mechanism offers no advantage
 and the optimal method  reduces to the Banach-Picard iteration from the outset. In contrast, for $\rho \in (1/2,1)$, there is a nontrivial initial phase during which Halpern iteration strictly dominates Banach-Picard. The length of this  phase becomes progressively longer as $\rho$ gets close to  1, with a substantial acceleration in the near-nonexpansive regime $\rho\approx 1$,
where  the geometric rate $\rho^n$ is slower than the $O(1/n)$ decay characteristic of Halpern iteration for nonexpansive mappings. This speedup is analyzed in more detail in the next section.

\vspace{-1ex}
\subsubsection{Comparison with Banach-Picard iterations} 
\label{secbp}
In order to compare the optimal Halpern bound $\|x^n\!-Tx^n\|\leq\kappa_0\,R_n^*$ with Banach-Picard, we observe that the latter satisfies $\|x^n\!-Tx^n\|\leq \kappa_0\,\rho^n$, and then  we may just  compare $\rho^n$ with $R_n^*=R_n^*(\rho)$.

First,  note that when minimizing $R_n^*$ the choice $\beta_n\equiv 1$ is feasible, so that $R_n^*\leq \rho^n\!$.
Moreover, if we denote $n_0=n_0(\rho)$ the smallest integer 
with $R_{n_0}^*\leq 1/\rho-1$, then for $n\geq n_0$ we have $R_n^*=R_{n_0}\rho^{n-n_0}$  and  
$\rho^n/R_n^*=\rho^{n_0}/R_{n_0}^*\ge \rho^{n_0+1}/(1\!-\!\rho)$. We claim that the latter expression diverges as $\rho\uparrow 1$, thereby showing that Halpern exhibits an increasing speedup with respect to the classical Banach--Picard iteration.

To substantiate this claim, consider the sequence 
$z_{n+1}\!=\frac{1}{4}(1+z_{n})^2$ with $z_0\!=0$.
For $n\leq n_0$, the recursion 
$R_n^*=\Vopt(R_{n-1}^*)$ implies $R_n^*=1-z_n/\rho$, and then
$n_0(\rho)$ can be characterized as the smallest $n\in\NN$ such that 
$\rho\leq\rho_n\!\triangleq\!\frac{1}{2}(1\!+\!z_n)$.
These $\rho_n$'s satisfy $\rho_{n+1}\!=\!\frac{1}{2}(1\!+\!\rho_{n}^2)$ with  $\rho_0\!=\!\frac{1}{2}$, and are strictly increasing to 1.
Hence, setting $\rho_{-1}=0$, we have that $n_0(\rho)\equiv n$ for all $\rho\in (\rho_{n-1},\rho_{n}]$, and consequently $\rho^{n_0(\rho)}\!=\rho^{n}\ge\rho_{n-1}^{n}$. 
Now, Lemma \ref{Le:Mon0} in  Appendix \ref{AppendixB} shows that $\rho_n^n$ decreases to $e^{-2}\!$, and therefore $\rho^{n_0(\rho)}\ge \rho\, e^{-2}$. This readily implies that $\rho^{n_0(\rho)+1}/(1\!-\!\rho)$ diverges as $\rho\uparrow 1$.

\begin{remark}The comparison above might not be fair since Banach-Picard bound holds with a smaller factor $r_0=\|x^0\!-Tx^0\|$ instead of $\kappa_0$, and we have 
$r_0\leq\kappa_0\leq r_0/(1\!-\!\rho)$. In fact, if $T$ has a fixed point $x^*\!$, then  $\kappa_0\geq \lim_{n\to\infty}\|x^0\!-Tx^n\|=\|x^0\!-x^*\|$ and the latter can be as large as $r_0/(1\!-\!\rho)$ ({\em e.g.} $Tx=x^0\!+\rho\, x$ with $x^0\!\neq 0$), so that 
 for $\rho\approx 1$ there could be a substantial gap between $r_0$ and $\kappa_0$. However, such large gaps can only arise on unbounded domains; otherwise $r_0$ and $\kappa_0$ are comparable. They may even coincide such as when the initial point satisfies $\|x^0\!-Tx^0\|=\diam(C)=\kappa_0$, or when $\|x^0\!-Tx^0\|=(1\!+\!\rho)\|x^0\!-x^*\|=\kappa_0$ ({\em e.g.} $Tx=-\rho\,x$). Notice also that the tight example in Appendix \ref{AppendixA} also satisfies $r_0=\kappa_0$.
 In such cases the preceding comparison  is exact. 
\end{remark}
\subsubsection{Comparison with optimal Halpern in Hilbert spaces} 
\label{secpr}
For nonexpansive maps, the optimal error
bounds for Halpern in normed spaces are known to be at most 4 times larger than the optimal bounds in Hilbert spaces \cite{cc2023,lie2021}. A similar estimate can be obtained when $\rho<1$, by considering the ratio between
the optimal bound  $\kappa_0 R_n^*$ in normed spaces  to Park-Ryu's optimal bound  in Hilbert spaces $\mbox{\sc pr}_n\triangleq \|x^0\!-x^*\|\,\rho^n\,(1-\rho^2)/(1-\rho^{n+1})$. Taking $\kappa_0 =(1+\rho)\|x^0\!-x^*\|$,  this ratio is

$$Q_n(\rho)\triangleq \frac{\kappa_0 R_n^*}{\mbox{\sc pr}_n}=\frac{(1-\rho^{n+1})}{(1-\rho)}\frac{R_{n}^*(\rho)}{\rho^n}.
$$

If we fix $n$ and let $\rho\to 1$ we get 
$R_n^*(\rho)\to R_n^*(1)$ and $\mbox{\sc pr}_n(\rho)\to 2/(n+1)$,
so that \cite[Theorem 2]{cc2023} yields  
$Q_n(1)\leq 4$ for all $n\in\NN$. 
However, if we exchange the order of the limits and first take 
$n\to\infty$ and then let $\rho\uparrow 1$, we observe that these ratios increase with $n$ and may become as large as $e^2$ (see Figure \ref{figura2}).
 Denoting as before $n_0=n_0(\rho)$ the smallest integer $n\in\NN$ satisfying $R_n^*(\rho)\leq 1/\rho-1$, for $n\geq n_0(\rho)$ we have $R_n^*(\rho)=R_{n_0}^*(\rho)\,\rho^{n-n_0}$, and therefore $Q_n(\rho)$ converges to
$$Q_\infty(\rho)=\frac{R_{n_0}^*(\rho)}{(1\!-\!\rho)\,\rho^{n_0}}.$$ 

\begin{figure}[ht!]
 \captionsetup{width=\linewidth}
\centering
\includegraphics[scale=1.25]{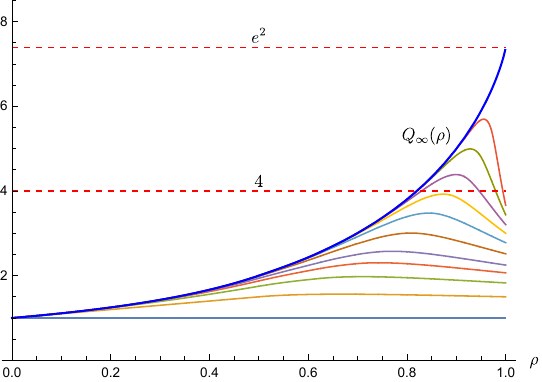}
\caption{\footnotesize Monotone convergence of $Q_n(\rho)$ towards  $Q_\infty(\rho)$ ($n=0,1,2,3,4,6,9,13,19,32,64$).
Each map $\rho\mapsto Q_n(\rho)$ increases  to a maximum and then falls below 4 as $\rho\uparrow 1$. The limit $Q_\infty(\rho)$ increases throughout $[0,1)$ and converges to $e^2$ when $\rho\uparrow 1$. }
\label{figura2}
\end{figure}

\begin{proposition}\label{Prop:Four}
For all $n\in\NN$ and $\rho\in[0,1)$  we have $Q_n(\rho)\leq Q_\infty(\rho)\leq e^2$. 
\end{proposition}

\begin{proof} By Lemma \ref{Le:Mon1} in Appendix 
\ref{AppendixB}, $Q_n(\rho)$ increases with $n$ so that $Q_n(\rho)\leq Q_\infty(\rho)$, while Lemma \ref{Le:Mon2} shows that $\rho\mapsto Q_\infty(\rho)$ is continuous and monotone increasing, so it suffices to find a sequence
$\rho_n\to 1$ such that $Q_\infty(\rho_n)\to e^2$. Take $z_n$ and $\rho_n$ as in \S{}\ref{secbp},
so that $n_0(\rho)\equiv n$ and
$Q_\infty(\rho)=\frac{\rho-z_n}{1-\rho}\frac{1}{\rho^{n+1}}$ for all  $\rho\in(\rho_{n-1},\rho_{n}]$.
In particular, for $\rho=\rho_n$ we have $\rho_n-z_n=1-\rho_n$
and therefore $Q_\infty(\rho_n)=1/\rho_n^{n+1}$ which converges to $e^2$ as shown in Lemma \ref{Le:Mon0}.
\end{proof}

\begin{remark} Notice that $Q_n(\rho_n)$ also converges to $e^2$ so that this constant is tight. Moreover,
{\em Proposition~\ref{Prop:Four}} implies that the bound $R_n^*$ is not only optimal and tight among all Halpern's iterations, but it is also within a factor $e^2$ of the best that can be achieved by any algorithm whose iterates belong to the span of previous residuals. Indeed,  as shown in  {\em \cite[Lee and Ryu]{lr2023}} there is a $\rho$-contractive map in the $\|\cdot\|_\infty$ norm
such that no algorithm satisfying this span condition can get below Park-Ryu's bound $\mbox{\sc pr}_n$.
\end{remark}

\subsection{The minimal displacement bound { (\texorpdfstring{$\rho>1$}{})}}
As in the case where $\rho=1$, when $\rho>1$ the minimum $R_n^*$ is attained at $\beta_n^*=\beta_\rho(R_{n-1}^*)$ for all $n\geq 1$, so that  $R_n^*=1-\frac{1}{4}\rho(1/\rho+1-R_{n-1}^*)^2$ and $\beta_{n+1}^*=\frac{1}{2\rho}(1+(\rho\beta_n^*)^2)$. 
The sequence $R_n^*$ decreases towards $r_\rho=1-1/\rho$ and $\beta_n^*$ increases to $\beta^*=1/\rho$. It follows that $\beta_n^*\rho\uparrow 1$ so that the optimal Halpern makes the best effort to make $T$ nonexpansive
and 
attains the so-called {\em minimum displacement bound}.
Namely, \cite[Goebel]{g1973} proved that if $C$ is a nonempty bounded  convex subset of a normed space $(X,\|\cdot\|)$, every $\rho$-Lipschitz map $T: C \to C$ with $\rho\geq 1$ satisfies 
\begin{equation}\label{eq:disp}
\inf_{x \in C }\|x -Tx\|\leq \diam(C)\left (1-1/\rho\right ),
\end{equation}
and gave an example of a map
that attains this with equality (see below).
Combining Proposition \ref{Teo2} and Proposition \ref{Prop_2.1}, with
 $\kappa=\diam(C)$, it follows that the minimax-optimal bounds $R_n^*$ attain asymptotically this minimal displacement, which is the best one can expect. Furthermore, we can also estimate the convergence rate of $R_n^*\downarrow \r$. Indeed, letting $e_n=\frac{\rho}{4}(R_n^*-\r)$, the recursion $R_n^*=1-\frac{1}{4}\rho(1/\rho+1-R_{n-1}^*)^2$ yields the logistic iteration
$e_n=e_{n-1}(1-e_{n-1})$ with $e_0=\frac{1}{4}$,  for which it is known that 
$\frac{1}{(n+3)+\ln(n+3)}\leq e_n\leq\frac{1}{n+3}$.

\begin{example}{\em (\cite[Goebel]{g1973})}. Consider $(\mathcal{C}[0,1],\|\cdot\|_\infty)$ the space of continuous functions, and  the convex set $C$ of all  $x\in\mathcal{C}[0,1]$ such that $0=x(0)\leq x(t)\leq x(1)=1$
for all $t\in[0,1]$, 
whose diameter is   $\diam(C)=1$. For each $\rho>1$ the map
$T: C \to C$ defined by 
$$
Tx(t)= \rho \max \big\{x(t)-1+1/\rho,0 \big\}$$
is $\rho$-Lipschitz and satisfies $\|x\!-Tx\|_\infty= 1-1/\rho$ for all $x\in C$, and, consequently, \eqref{eq:disp} holds with equality. In particular, starting from $x^0 \in C$, every Mann's iterates have constant residuals $\|x^n\!-Tx^n\|_\infty\equiv 1-1/\rho$.
\end{example} 

\subsection{Halpern iteration with adaptive anchoring}
\label{sec2.4}

As noted before Theorem \ref{Thm:Main}, the sequences $\beta_n^*=\bopt(R_{n-1}^*)$ and $R_n^*=\Vopt(R_{n-1}^*)$  only depend on $\rho$ and are agnostic to the specific map $T\in\Lip(\rho)$ being considered,
while the bound 
$\|x^n\!-Tx^n\|\leq\kappa_0\, R_n^*$ 
 also involves the {\em a priori} estimate $\kappa_0$. 
Although these bounds are optimal in a minimax sense, for a given $T$ one could hope  to get a faster algorithm by using  anchoring parameters $\beta_n$ adapted to this specific map and, at the same time, to avoid the need for the {\em a priori} estimate $\kappa_0$.

A close inspection of the proof of Proposition \ref{Prop_2.1} suggests to consider the quantity $R_n=1-\beta_n+\beta_n\|Tx^n-Tx^{n-1}\|/\hat\kappa_n$ with $\hat \kappa_n=\max\{\|x^0\!-Tx^i\|:0\leq i\leq n\}$, which readily gives the error bound
$$\|x^n\!-Tx^n\|\leq (1-\beta_n)\|x^0\!-Tx^n\|+\beta_n\|Tx^n\!-Tx^{n-1}\|\leq\hat\kappa_n R_n.$$
Proceeding as in the proof of Proposition \ref{Prop_2.1}, and using the definition of $\hat\kappa_n$, we get
\begin{align*}
\|x^n\!-x^{n-1}\|
&\leq |\beta_{n-1}-\beta_n|\hat\kappa_{n}+\min\{\beta_n,\beta_{n-1}\}\,\|Tx^{n-1}\!-Tx^{n-2}\|
\end{align*}
and since $\|Tx^{n-1}-Tx^{n-2}\|=\hat\kappa_{n-1}(R_{n-1}-1+\beta_{n-1})/\beta_{n-1}$ with $\hat\kappa_{n-1}\leq\hat\kappa_n$, it follows  
\begin{align*}
R_n&\leq 1-\beta_n+\beta_n\rho\|x^n-x^{n-1}\|/\hat\kappa_n\\
&\leq 1-\beta_n+|\beta_{n-1}-\beta_n|\beta_n\rho+\min\{\beta_n,\beta_{n-1}\}\beta_n\rho(R_{n-1}-1+\beta_{n-1})/\beta_{n-1}.
\end{align*}
When $\beta_n\geq\beta_{n-1}$ this simplifies to
\begin{equation}\label{eq:ada}
R_n\leq 1-\beta_n+\rho\beta_n^2+\beta_{n}\rho(R_{n-1}-1)
\end{equation}
whose right hand side is precisely the quadratic that is minimized in $\Vopt(R_{n-1})$. This suggests taking $\beta_n=\bopt(R_{n-1})$
the corresponding minimizer, updating $R_{n}$ along the iterations
using the formula suggested above. Below we provide a pseudo-code for this adaptive Halpern iteration.

\begin{algorithm}[ht!]
  \caption{\mbox\adahalpern{} (for  $T:C\to C$  a $\rho$-Lipschitz map)}
  \begin{algorithmic}
    \STATE select $x^0\in C$, and set $R_0=1$ and $\hat\kappa_0=\|x^0\!-Tx^0\|$
    \FOR{ $n=1,2,\ldots$}
      \STATE compute $\beta_n=\bopt(R_{n-1})$
      \STATE update $x^n=(1-\beta_n)x^0+\beta_n\,Tx^{n-1}$
      \STATE update $\hat\kappa_{n}=\max\{\hat\kappa_{n-1},\|x^0\!-Tx^n\|\}$
       \STATE update $R_n=1-\beta_n+\beta_n\|Tx^n-Tx^{n-1}\|/\hat\kappa_n$
    \ENDFOR
  \end{algorithmic}
\end{algorithm}

\begin{theorem}\label{Thm:ada_halpern}
Let $T\in\Lip(\rho)$ and consider the iteration \adahalpern{}.
Then, for all $n\geq 1$ we have $\beta_n\geq\beta_{n-1}$ and $r_\rho\leq R_n\leq \Vopt(R_{n-1})$, and the following error bound is satisfied $\|x^n-Tx^n\|\leq\hat\kappa_n R_n\leq \hat\kappa_n R_n^*$, with $R_n^*$ the minimax optimal bound  in the previous section. 
\end{theorem}

\begin{proof}
Using \eqref{eq:ada} and  $\Vopt(r)\leq r$ we get inductively that $R_{n} \le V_{\rho}(R_{n-1}) \le R_{n-1}$ for all $n\geq 1$, and since $r\mapsto \bopt(r)$ is decreasing it follows that  $\beta_n\geq\beta_{n-1}$.
Moreover, since $R_0=R_0^*$ and $\Vopt(\cdot)$ is
increasing, it also follows inductively that 
$R_n\leq \Vopt(R_{n-1})\leq \Vopt(R_{n-1}^*)=R_n^*$.

The inequality $ R_n\geq r_\rho$ is trivial for $\rho\leq 1$ since $r_\rho=0$. For $\rho>1$ we have $r_\rho=1-1/\rho$ so that $R_0=1\geq r_\rho$. Inductively, if  $R_{n-1}\geq r_\rho$  then $\beta_n=\bopt(R_{n-1})\leq \bopt(r_\rho)=1/\rho$ and 
 the definition of $R_n$ yields $R_n\geq 1-\beta_n\geq 1-1/\rho=r_\rho$. 
\end{proof}

\begin{remark}  The following two observations are in order.\\[0.5ex]
$a)$ In the Hilbert setting, an adaptive Halpern iteration was already proposed in
 {\em \cite[He {\em et al.}]{he2024convergence}}. The same iteration was also proposed in 
{\em \cite[Suh {\em et al.}]{spr2023} }
for solving maximal monotone inclusions reformulated as a fixed point of the corresponding reflection operator. The iteration
 relies on the Hilbert structure and 
is different from the \adahalpern{} method, which is designed to work on general normed spaces.\\[0.5ex]
$b)$
Although \adahalpern{} does not require $\kappa_0$, the iteration is not parameter free as it needs an estimate of the Lipschitz constant $\rho$.
In this direction we mention that a parameter-free Halpern-type method was developed by
{\em  \cite[Diakonikolas]{diakonikolas2020halpern}} in Hilbert spaces. More recently,  {\em  \cite[Diakonikolas]{d2025}} reconsidered this question for Lipschitz maps with bounded domains
in general normed spaces: when $\rho \le 1$ it presents a fully parameter-free algorithm, while for $\rho > 1$ one that only requires a bound for the diameter of the domain.
\end{remark}

\subsubsection{A simple numerical illustration}
\label{sec:numil}
Theorem \ref{Thm:ada_halpern} shows that \adahalpern{} guarantees an error bound $\hat\kappa_nR_n$ at least as good as the minimax optimal bound $\kappa_0R_n^*$ ensured by \opthalpern{}. However, the potential advantages of these iterations with respect to \bpicard{} can be better judged by their empirical performance. 
In order to build some intuition, we consider a toy example of a linear map on $(\RR^2,\|\cdot\|_\infty)$ which combines a rotation with a contraction, namely $Tx=\rho\, A x$ with $\rho\in(0,1)$ and 
$$A=
\frac{1}{|\cos\theta|+|\sin\theta|}
\begin{pmatrix}
\cos\theta & -\sin\theta\\
\sin\theta & \cos\theta
\end{pmatrix}
$$

This simple map is convenient because its Lipschitz constant is known exactly, namely, it is $\rho$-Lipschitz for the infinity norm $\|\cdot\|_\infty$. Although one cannot draw any general conclusions from this very special case, the simulations confirm the theoretical results. 

Figure \ref{figura3} compares the evolution of the residuals $\|x^n\!-Tx^n\|_\infty$ for all three iterations in two scenarios. The left panel considers the angle $\theta=\pi/2$ whereas the right panel is for $\theta=\pi/4$. 
The contraction constant is fixed to $\rho=0.98$ which is very close to 1. 

\begin{figure}[ht!]
  \centering
\captionsetup{width=\linewidth}
  \begin{minipage}{.595\linewidth}
    \centering
    \includegraphics[width=\linewidth]{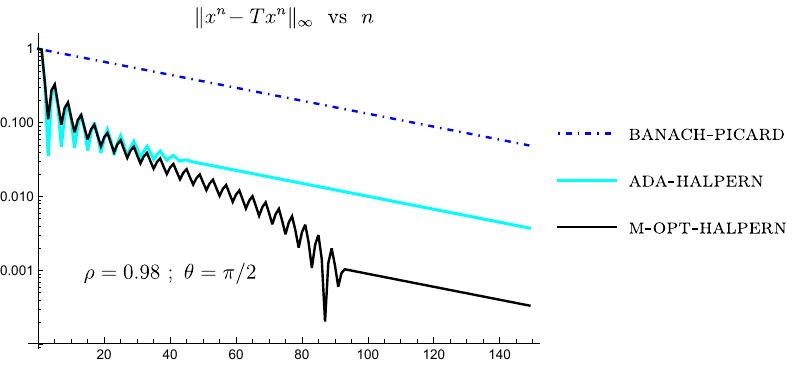}
  \end{minipage}\hfill
  \begin{minipage}{.405\linewidth}
    \centering
    \includegraphics[width=\linewidth]{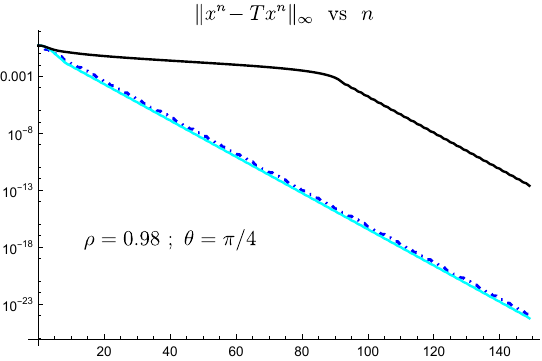}
  \end{minipage}
  \caption{\footnotesize
  Comparison of \opthalpern{} and \adahalpern{} with \bpicard{} (in log-scale).}
  \label{figura3}
\end{figure}

For $\theta=\pi/2$, the best performance 
is achieved by \opthalpern, followed by \adahalpern{}. Both methods exhibit an initial phase in which they outperform \bpicard{}, before eventually reverting to a simple Banach-Picard iteration. This initial phase is longer for \opthalpern{} which explains its advantage. 
Notice that if we had taken $\rho=1$, then 
$T$ would reduce to a pure rotation by an angle $\pi/2$. In this case, \bpicard{} would maintain constant residuals and make no progress, whereas \opthalpern{} would still exhibit a decay of order $O(1/n)$ without ever switching to Banach-Picard. The oscillations observed in these plots arise frequently in fixed-point iterations applied to mappings involving rotations.

The right panel with $\theta=\pi/4$ shows a different picture. While $T$ remains
$\rho$-Lipschitz for $\|\cdot\|_\infty$, it is also $\tilde\rho$-Lipschitz for $\|\cdot\|_2$ with $\tilde\rho=\rho/\sqrt{2}$. Since \opthalpern{} is based on the $\|\cdot\|_\infty$ framework, it cannot exploit the smaller $\tilde\rho$ and takes nearly 100 iterations before switching to Banach-Picard. In contrast, the \bpicard{} iteration decreases geometrically from the outset. Interestingly,  although \adahalpern{} is also based on the $\|\cdot\|_\infty$ framework, it detects quite early the advantage of switching to {\sc banach-picard}, yielding a marginal yet measurable improvement.
Notice that the angle $\theta=\pi/4$ is the most favorable scenario for \bpicard{} as it achieves the smallest Lipschitz constant $\tilde\rho$ for the Euclidean norm $\|\cdot\|_2$. In constrast,
$\theta=\pi/2$ is the worst scenario
for \bpicard{} with $\tilde\rho=\rho$.

 \section{Optimal Halpern based on alternative bounds}\label{sec3}
 
The minimax-optimal bounds in \S\ref{sec2} were  based on  {\em a priori} estimates \eqref{eq:bkappa} such as  $\kappa=\diam(C)$ or $\kappa=(1\!+\!\rho)\|x^0\!-x^*\|$ (for $\rho\leq 1$ and $x^*$ a fixed point).
We now consider the case where $C$ is a possibly unbounded domain,
using alternative error bounds based directly on some estimate $\delta_0\ge \|x^0\!-x^*\|$. We allow $\rho>1$ but we assume that
\begin{equation}\label{eq:fixed_point}\tag{$\textsc{f}$}
\text{there exists\,} x^* \in C  \,\,\text{such that}\,\, x^*\!=Tx^*.
\end{equation}

\subsection{An alternative recursive bound for nonlinear maps}\label{sec3.1}

In what follows we assume that  $(\beta_n)_{n\in\NN}$ is non-decreasing  with $\beta_0=0$ and $\beta_n\in [0,1]$. Notice that, in contrast with \S\ref{sec2}, we now allow $\beta_n=0$.
We  consider the alternative recursive bounds $\dbis_n$ and $\Rbis_n$ given by
$$(\forall\,n\in\NN)\quad\left\{\begin{array}{l}
    \mu_n\;=\,1\!-\!\beta_n+\rho\, \beta_n\, \mu_{n-1}\\
    \nu_n\,=\,1+\rho\, \mu_n
    \\
    \dbis_n\,=\,(\beta_{n}\!-\!\beta_{n-1})\nu_{n-1}+\rho\,\beta_{n-1}\dbis_{n-1}\\[0.2ex]
    \Rbis_n=(1-\beta_n)\nu_n+\rho\,\beta_n\dbis_n
\end{array}\right.$$
with $\dbis_{-1}$ and $\mu_{-1}$ chosen arbitrarily (their values are irrelevant since $\beta_0=0$). 

\begin{proposition}\label{Prop_2.1Bis}
Assume \eqref{eq:fixed_point}, $(\beta_n)_{n\in\NN}$ non-decreasing with $\beta_0=0$, and let $\delta_0\geq\|x^0\!-x^*\|$. Then, for all $n\in\NN$ we have $\|x^n-x^*\|\leq \delta_0\mu_n$, $\|x^0\!-Tx^n\|\leq \delta_0\nu_n$, $\|x^n\!-x^{n-1}\|\leq \delta_0 \dbis_n$, and $\|x^n\!-Tx^n\|\leq\delta_0\Rbis_n$. Moreover, choosing $\Rbis_{-1}$ arbitrarily,  we have 
\begin{equation}\label{eq:Rec1}
\mbox{$\Rbis_n=(1\!+\!\rho)-(1\!+\!3\rho)\beta_n+2\rho\,\beta_n^2+\rho\,\beta_n\Rbis_{n-1}.$}
\end{equation}
\end{proposition}
\begin{proof}
   See Appendix~\ref{ApC1}.
\end{proof}

This result suggests to select $\beta_n$ by minimizing recursively these alternative $\Rbis_n$'s. Namely, for $r\geq 0$ we consider the one-dimensional quadratic 
minimization problem 
\[\Voptbis(r)\!\triangleq\! \min_{\beta\in [0,1]}~ \mbox{$(1\!+\!\rho)-(1\!+\!3\rho)\beta+2\rho\,\beta^2+\rho\,\beta\, r$}\]
which is attained at 
$$\boptbis(r)=\left\{\begin{array}{cl}
1&\mbox{if }  r \leq 1/\rho-1,\\
\beta_\rho(r)&\mbox{if } r\in[1/\rho-1, 1/\rho+3]\\
0&\mbox{if }r\geq 1/\rho+3. 
\end{array}\right.
$$
with $\beta_\rho(r)\triangleq(1/\rho+3-r)/4$
 the unconstrained minimizer. Consequently
$$\Voptbis(r)=\left\{\begin{array}{cl}
\rho \,r&\mbox{if }  r \leq 1/\rho-1,\\
(1+\rho)-2\rho (\beta_\rho(r))^2&\mbox{if } r\in[1/\rho-1, 1/\rho+3]\\
(1+\rho)&\mbox{if }r\geq 1/\rho+3. 
\end{array}\right.
$$
This yields the following optimized variant of Halpern.

\begin{algorithm}[ht!]
  \caption{\mbox\flathalpern{} (for   $T:C\to C$ a $\rho$-Lipschitz map)}
  \begin{algorithmic}
    \STATE select $x^0\in C$ and set $\Rbis_0=1\!+\!\rho$ 
    \FOR{ $n=1,2,\ldots$}
      \STATE compute  $\betabis_n=\boptbis(\Rbis_{n-1})$ and $\Rbis_n=\Voptbis(\Rbis_{n-1})$
      \STATE update $x^n=(1\!-\!\betabis_n)\,x^0+\betabis_n\,Tx^{n-1}$
    \ENDFOR
  \end{algorithmic}
\end{algorithm}

The behavior of \flathalpern{} varies depending on $\rho$.
Note that $r\mapsto \boptbis(r)$ is non-increasing, while $r\mapsto \Voptbis(r)$ is non-decreasing with a fixed point $\rbis=\Voptbis(\rbis)$ that satisfies $\rbis\leq 1+\rho$, namely

$$\rbis=\left\{
\begin{array}{cl}
0&\mbox{if }\rho<1\\
(\sqrt{2}+\!1)^2(1-1/\rho)&\mbox{if }\rho\in[1,\sqrt{2}+\!1]\\
1+\rho&\mbox{if }\rho>\sqrt{2}+\!1.
\end{array}
\right.$$
\vspace{-0.5ex}

\noindent Moreover, for all $r>\rbis$ we have $\Voptbis(r)<r$,  so that the sequence $\Rbis_n=\Voptbis(\Rbis_{n-1})$ with $\Rbis_0=1+\rho$, decreases towards $\rbis$ and therefore $\betabis_n$ increases to $\betabis_{\!\rho}=\boptbis(\rbis)$ with
\vspace{-1ex}

$$\betabis_{\!\rho}=\left\{
\begin{array}{cl}
1&\mbox{if }\rho<1\\
(\sqrt{2}+\!1-\rho)/(\rho\sqrt{2})&\mbox{if }\rho\in[1,\sqrt{2}+\!1]\\
0&\mbox{if }\rho>\sqrt{2}+\!1.
\end{array}
\right.$$
These observations  directly imply the next result, stated without proof.

\begin{theorem}\label{Thm:unbounded}
Let $T\in\Lip(\rho)$ and consider the iteration \flathalpern{}.
Then, for all $n\geq 1$ we have $\|x^n\!-Tx^n\|\leq\delta_0\Rbis_n$ 
with $\Rbis_n\searrow \rbis$ 
and $\betabis_n\nearrow\betabis_{\!\rho}$.
\end{theorem}

The following remarks hold for the different ranges of $\rho$.

\vspace{1.5ex}
\noindent \underline{\sc Case $\rho< 1$}: In this range, $\Rbis_n\downarrow 0$ and
\flathalpern{} features an initial phase where $\betabis_n\!<\!1$, and then switches to Banach-Picard as soon as $\Rbis_n\leq 1/\rho-1$, with $\betabis_n\equiv 1$ and 
 $\Rbis_{n+1}=\rho\,\Rbis_{n}$, which converges geometrically to $0$. When $\rho\leq \sqrt{2}\!-\!1$, the initial $\Rbis_0=1+\rho$ is already smaller than the threshold $1/\rho-1$ so that $\betabis_n\equiv 1$ from the outset, whereas for  $\rho\in(\sqrt{2}\!-\!1,1)$ the initial phase of \flathalpern{} strictly dominates Banach-Picard, with an 
error bound $\delta_0\Rbis_n$ that is also smaller than the
 $\kappa_0R_n^*$  guaranteed by  \opthalpern{} with $\kappa_0\!=\!(1\!+\!\rho)\delta_0$. Notice that the acceleration of \opthalpern{} over Banach-Picard  occurs on the smaller interval $\rho\in(\frac{1}{2},1)$.

\vspace{1.5ex}
\noindent \underline{\sc Case $\rho=1$}: Here $\Rbis_n\equiv 2\Ropt_n$ and $\betabis_n\equiv\beta_n$,
so that \flathalpern{} coincides with \opthalpern, and both correspond to the iteration studied in \cite[Section 4]{cc2023}. 

\vspace{1.5ex}
\noindent \underline{\sc Case $\rho\in (1,\sqrt{2}+\!1)$}:
In this case, $\Rbis_n$ decreases to $\rbis$ which is strictly larger than
the minimal displacement bound $\r=1-1/\rho$ attained by \opthalpern{}. However, we recall that here we do not assume a bounded domain $C$ but only the existence of a fixed point. On the other hand, $\betabis_n$ increases towards
$\betabis_{\!\rho}$ which is strictly smaller than the limit $1/\rho$ attained in \S{}\ref{sec2}; so that the limit map $T^{\flat}_{\!\rho}x\triangleq (1-\betabis_{\!\rho})x^0+\betabis_{\!\rho}\,Tx$ is a contraction with constant
$L=\rho\,\betabis_{\!\rho}<1$; and the iterates of \flathalpern{} converge to its unique fixed point $x^\flat_{\rho}$
with $\|x^n\!-x^\flat_\rho\|\leq (\sqrt{2}+\!1)\delta_0(n\!+\!1) L^n$ (see Appendix \ref{ApC2}).

\vspace{1.5ex}
\noindent \underline{\sc Case $\rho\geq\sqrt{2}+\!1$}: In this regime $\Rbis_n\equiv 1\!+\!\rho$ and $\betabis_n\equiv 0$, so that the iterates $x^n\equiv x^0$ remain at the initial point. Hence, within this parameter range, \flathalpern{} offers no benefit compared to \opthalpern, which remains effective in attaining the minimal displacement bound.

\vspace{2ex}

\subsection{Minimax-optimal Halpern  for affine maps}
Consider now the case where $T:X\to X$ is an affine $\rho$-Lipschitz map $Tx=Ax+b$ with $A$ linear, and which has some fixed point $x^*\in\Fix(T)$. 
Denoting $\B_i^n=\prod_{j=i}^n\beta_j$ for $i=1,\ldots n$, and adopting the conventions $\B_i^n=0$ for $i\leq 0$ and $\B_{i}^n=1$ for $i>n$, we get inductively (see \cite[Section 4.2]{cc2023} for details)
$$    x^n\!-Tx^n
      =\mbox{$\sum_{i=0}^{n+1}(-\B^n_{i+1}\!+\!2\B_i^n\!-\!\B_{i-1}^n)\,\, A^{n+1-i}(x^0\!-x^{*})$}
$$
where $A^m$ stands for the $m$-fold composition of $A$. Now, taking $\delta_0\geq\|x^0-x^*\|$ we obtain $\|x^n\!-Tx^n\|\leq\delta_0\, L_n(\beta)$ with
$$ L_n(\beta)=
     \mbox{$ \sum_{i=0}^{n+1}|\B_{i+1}^n\!-\!2\B_i^n\!+\!\B_{i-1}^n|\,\rho^{n+1-i}
    $}.
$$
\newpage
\begin{example} The following examples show that this upper bound $L_n(\beta)$ is tight.\ \\[1ex] 
$a)$ Let $T: \ell^{1}(\mathbb{N}) \rightarrow \ell^{1}(\mathbb{N})$ be the $\rho$-scaled right-shift 
$$T(x_0, x_1, x_2, \dots) =\rho\, (0, x_0, x_1,x_2, \dots),$$ 
with unique fixed point $x^{*}=(0,0,0,\dots)$. For $x^0=(1,0,0,\dots)$ we have $\|x^0-x^*\|_1=1=\delta_0$ and Halpern's iterates satisfy $\|x^n-Tx^n\|_1= L_n(\beta)$ for all $n\in\NN$.
\\[1ex]
$b)$ 
Let $T: \mathbb{R}^{n+2} \rightarrow \mathbb{R}^{n+2}$ in dimension $d=n+2$ be given by
$$T(x_0,x_1, x_2, \dots, x_{n},x_{n+1}) =\rho\, (x_{n+1}, x_0,x_1, \dots,  x_{n-1},x_n),$$
with fixed point $x^{*}=(0,\dots,0)$. Consider Halpern's iterates started from $x^0$ with  $x^0_0=-1$, $x^0_{n+1}=1$, and $x^0_{i}=\mathop{\rm sign}\para{-\B_{i+1}^n\!+\!2\B_i^n\!-\!\B_{i-1}^n}$ for $i=1,\dots,n$, so that $\|x^0\!-x^*\|_\infty=1=\delta_0$. The $(n\!+\!1)$-th coordinate of $T^{n-i+1}x^0$ is $\rho^{n-i+1}x^0_i$ so that
$(x^n\!-Tx^n)_{n+1}=L_n(\beta)$.
    Hence $\|x^n-Tx^n\|_\infty\geq L_n(\beta)$ and since we showed above that the reverse inequality is always satisfied, we get
    $\|x^n-Tx^n\|_\infty=L_n(\beta)$. Note that in this case the tightness is  ensured at iteration $n$.
\end{example}
\vspace{-0.5ex}

In order to design an optimal iteration  we  minimize 
$L_n^*=\min_{\beta} L_n(\beta)$.
For $\rho=1$ this was solved in \cite{cc2023}. Below we describe  the cases $\rho<1$ and $\rho>1$. See Appendix \ref{ApC3} for the proof.

\begin{theorem}\label{Thm:affine}\hfill\ \\[1ex]
$(a)$
    Let $\rho<1$ and $n_0= \max \big\{k\in \NN: \frac{1+\rho^{k+1}}{k+1}\leq \rho\, \frac{1+\rho^{k}}{k}\big\}$. Then, the minimum of $L_n(\beta)$ is attained with $\beta_i=\frac{i}{i+1}$ for $i\leq n_0$ and $\beta_i=1$ for $i>n_0$, so that 
\begin{equation}\label{eq:optRn}
L_n^*=\left\{\begin{array}{ll}
\frac{1+\rho^{n+1}}{n+1}&\mbox{if }n\le n_0,\\[1.5ex]
\frac{1+\rho^{n_0+1}}{n_0+1}\rho^{n-k_0}&\mbox{if }n> n_0.
\end{array}\right .
\end{equation}
\\[1ex]
$(b)$ Let $\rho>1$ and $n_0= \max \big\{k\in \NN: \frac{1+\rho^{k+1}}{k+1}\leq \frac{1+\rho^{k}}{k}\big\}$. For $n\leq n_0$ the minimum of $L_n(\beta)$ is attained with $\beta_i=i/(i+1)$ for all $i\leq n$, whereas for $n>n_0$ the optimal Halpern performs $n_0$ iterations with these same $\beta_i$'s and then halts the iteration by setting $x^n\!=x^{n-1}\!$, so that 
\begin{equation}\label{eq:optRn_exp}
L_n^*=\left\{\begin{array}{ll}
\frac{1+\rho^{n+1}}{n+1}&\mbox{if } n\leq n_0,\\[1.5ex]
\frac{1+\rho^{n_0+1}}{n_0+1}&\mbox{if }n> n_0.
\end{array}\right .
\end{equation}
\end{theorem}

\begin{remark} Using 
the upper branch of the Lambert function $W_0(\cdot)$  we can explicitly determine $n_0$ as follows:\\[0.5ex]
$a)$ For $\rho<1$ we have 
$n_0=\lfloor\frac{\rho}{1-\rho}-\frac{1}{\ln \rho}W_0(\frac{\ln \rho}{\rho-1}\rho^{1/(1-\rho)})\rfloor$ which is non-decreasing in $\rho$, with\\ \hspace*{3.2ex}$n_0=0$ iff $\rho<\sqrt{2}\!-\!1$ and $n_0\to \infty$ when $\rho\uparrow 1$. 
\\[1ex]
$b)$ For $\rho>1$ we have
 $n_0=\lfloor\frac{1}{\rho-1}+\frac{1}{\ln \rho}W_0(\frac{\ln\rho}{\rho-1}\rho^{1/(1-\rho)})\rfloor$ which is non-increasing in $\rho$, with\\ \hspace*{3.2ex}$n_0=0$ iff $\rho> \sqrt{2}\!+\!1$ and $n_0\to\infty$ when $\rho\downarrow 1$.
\end{remark}

Below we include  a pseudo-code for the resulting iteration \affhalpern{} and a  comparison with \flathalpern{}, \opthalpern{}  and \bpicard{}, for the linear map $T_d(x_1,\ldots,x_d)=\rho\,(x_d,x_1,\ldots,x_{d-1})$
in $(\RR^d,\|\cdot\|_\infty)$ with $\rho<1$ and $\rho>1$
(see Figure~\ref{figura4}).

\begin{algorithm}[ht!]
  \caption{\mbox\affhalpern{} (for   $T:X\to X$ a $\rho$-Lipschitz affine map)}
  \begin{algorithmic}
  \STATE select $x^0\in C$
    \FOR{ $n=1,2,\ldots$}
      \STATE {\bf if} ~$\frac{1+\rho^{n+1}}{n+1}\leq \min\{\rho,1\}\frac{1+\rho^{n}}{n}$ {\bf then}
      \STATE ~~~$x^n=\frac{1}{n+1}\,x^0+\frac{n}{n+1}\,Tx^{n-1}$
      \STATE {\bf else} 
       \STATE ~~~{\bf if}  $\rho<1$ {\bf then} $x^n\!=Tx^{n-1}$ 
      {\bf else} $x^n\!=x^{n-1}$
    \ENDFOR
  \end{algorithmic}
\end{algorithm}

\vspace{2ex}

\begin{figure}[ht!]
  \centering
  \captionsetup{width=\linewidth}
  \begin{minipage}{.595\linewidth}
    \centering
    \includegraphics[width=\linewidth]{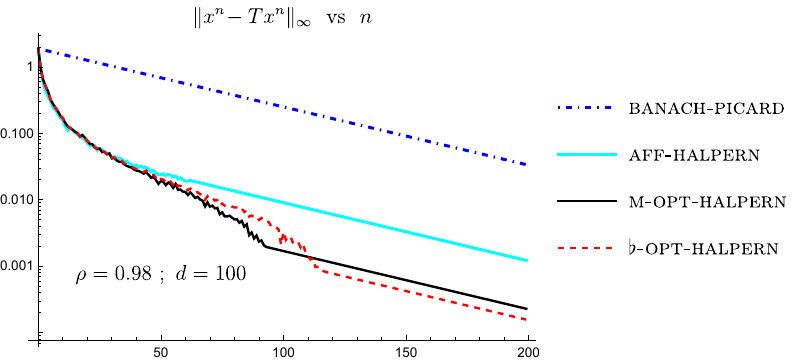}
  \end{minipage}\hfill
  \begin{minipage}{.405\linewidth}
    \centering
    \includegraphics[width=\linewidth]{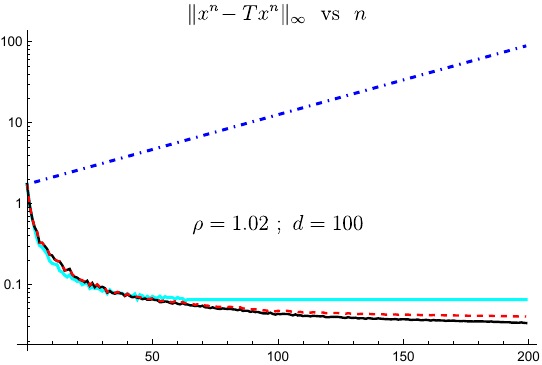}
  \end{minipage}
  \caption{\footnotesize Comparison of  \affhalpern{}  and \flathalpern{} with \bpicard{} and \opthalpern, for the map $T_d$ in dimension $d=100$, starting from a random initial point $x^0$ with  $x^0_i\sim U[-1,1]$. Multiple runs with $\rho=0.98$ show that the residual achieved by \flathalpern{} is consistently 2 orders of magnitude smaller than \bpicard{}, with even larger speedups  for $\rho$ closer to 1. Initially \opthalpern{}    converges slightly faster than \flathalpern, but later the order is eventually reversed (for $\rho>1$ this occurs for $n$ larger than the 200 iterations shown in the plot).}
  \label{figura4}
\end{figure}

\vspace{4ex}

{\bf Acknowledgements.} The work of M. Bravo and R. Cominetti was partially supported by FONDECYT Grant No.1241805. The work of J. Lee was supported by the Samsung Science and Technology Foundation (Project Number SSTF-BA2101-02).

\newpage

\bibliographystyle{ims}



\appendix

\section{Tight bounds for Mann iterates of Lipschitz maps}\label{AppendixA}
Let $T:C\to C$ be $\rho$-Lipschitz with $\rho>0$ and $C$ a convex subset of a normed space $(X,\|\cdot\|)$.
Given a triangular array of averaging scalars 
$\pi^n_i\ge 0$ with $\sum^n_{i=1}\pi^n_i=1$, $\pi^n_n>0$ and $\pi^n_i= 0$ for $i>n$,
the Mann iterates started from 
$x^0 \in C$ are defined by \vspace{-1ex}

\[(\forall n\in\NN)\quad x^n=\pi^n_0x^0+\sum^n_{i=1}\pi^n_i\,Tx^{i-1}.\]
\vspace{-0.5ex}

\noindent By convention we set $Tx^{-1}\!= x^0$ so that $x^n\!=\sum^n_{i=0}\pi^n_i\,Tx^{i-1}\!$, and we  let $\mathcal N\triangleq\NN\cup\{-1\}$.
As in Section \S\ref{sec2} we look for upper bounds for the fixed point residuals $\|x^n-Tx^n\|$, for which we require an {\em a priori} bound on the orbit. Namely, we assume the existence of a constant $\kappa$ such that
\begin{equation}\tag{$\textsc{b}_\kappa$}
 (\forall\,m,n\in \N)\quad\|Tx^m\!-Tx^n\|\leq\kappa.
\end{equation}
See \S\ref{sec2} for examples where this condition holds.

Adapting the analysis the results for  $\rho=1$ in \cite{bcc2022,bc2018,cc2023}, 
we introduce a nested family of optimal transport problems which will allow to derive bounds for the residuals. 

\begin{definition}
    Let  $\mathcal{F}_{m,n}$ denote the set of transport plans from $\pi^m$ to $\pi^n$, that is to say, all the vectors $z=(z_{i,j})_{i=0,\dots, m; j=0,\dots, n}$ with $z_{i,j} \ge 0$ and such that
$$\left\{\begin{array}{cl}
    \mbox{$\sum^n_{j=0} z_{i,j}$} = \pi^m_i &\mbox{ for all } i=0,\dots, m\\
        \mbox{$\sum^m_{i=0} z_{i,j}$} = \pi^n_j & \mbox{ for all } j=0,\dots, n.
\end{array}\right.$$
We define $R_n\triangleq\sum^n_{i=0} \pi^n_i\,c_{i-1,n}$ with $c_{m,n}=\min\{1,\rho\, d_{m,n}\}$ and $d_{m,n}$ defined  recursively by the nested family of optimal transports 
\[(\forall m,n\in\NN)\quad d_{m,n} = \min_{z \in \mathcal{F}_{m,n}}\sum^m_{i=0}\sum^n_{j=0}  z_{i,j}\,c_{i-1,j-1}\] 
starting with $d_{-1,n}=d_{n,-1}=1/\rho$ for all $n\in\NN$ and $d_{-1,-1}=0$.
\end{definition}

\begin{remark}
Since $c_{i-1,j-1}\leq 1$ it follows that $d_{m,n}\leq 1$ for $m, n \ge 0$. Hence,  when $\rho\leq 1$ we have   $c_{m,n}=\rho\, d_{m,n}$. Also, a simple induction argument shows that there is symmetry $d_{m,n}=d_{n,m}$ and that $d_{n,n}=0$ for all $n \geq 0$.
\end{remark}

Notice that the scalars $R_n$ as well as $c_{m,n}$ and $d_{m,n}$ depend solely on $\rho$ and the $\pi^n$'s. As shown next,
they provide universal bounds for Mann iterates 
for Lischitz maps in arbitrary normed spaces.
Later, in Theorem \ref{thm:tightLip}, we will show that these bounds are in fact tight.

\begin{proposition}\label{prop_6.2Lip}
 Let $(x^n)_{n\in\NN}$ be a Mann sequence for  $T\in\Lip(\rho)$ and assume that \eqref{eq:bkappa} holds.  Then $\|x^n-x^m\|\leq \kappa\,d_{m,n}$ and $\|x^n-Tx^n\|\le \kappa\,R_n$
for all $m,n\in\NN$.
\end{proposition}

\begin{proof} By rescaling the norm we may assume that $\kappa=1$. Take $m,n\in \NN$ and suppose  that we already have $\|x^{i}-x^{j}\| \le d_{i,j}$ for all $0\le i<m$ and $0\le j<n$, so that $\|Tx^{i}-Tx^{j}\|\leq\min\{\rho\, d_{i,j},1\}=c_{i,j}$. Also, from \eqref{eq:bkappa} we have $\|Tx^{i}-Tx^{-1}\|\leq\kappa=1=c_{-1,i}$ and  $\|Tx^{-1}-Tx^{j}\|\leq\kappa=1=c_{-1,j}$.
Using these estimates, for each transport plan $z\in \mathcal{F}_{m,n}$ we have 
$$   \|x^m-x^n\|  =\mbox{$ \left \| \sum^m_{i=0}\sum^n_{j=0} z_{i,j}(Tx^{i-1}-Tx^{j-1}) \right\| $} \le \mbox{$ \sum^m_{i=0}\sum^n_{j=0} z_{i,j}\, c_{i-1,j-1} $}$$
 and minimizing over $z\in \mathcal{F}_{m,n}$ we get $\|x^m-x^n\| \le d_{m,n}$. A double induction on $m$ and $n$ implies that this inequality holds for all $m,n \in \NN$, and  using
a triangle inequality we can bound the  residuals as
$$\mbox{$
    \|x^n-Tx^n\|= \left\|\sum^n_{i=0}\pi^n_i(Tx^{i-1}-Tx^n) \right\|
    \le \sum^n_{i=0} \pi^n_i\,c_{i-1,n} =R_n.$}
$$
\end{proof}
\begin{lemma}
    Both $d_{m,n}$ and $c_{m,n}$ define metrics over the set $\mathcal N$ with $c_{m,n}\in[0,1]$.
\end{lemma}

\begin{proof}
It suffices to show that $d_{m,n}$ is a metric.
Given that $\pi_n^n>0$ and $\rho>0$ we have, by induction, that $d_{m,n}>0$ for $m<n$.  It remains to establish the triangle inequality.
We proceed inductively by showing that  for each $\ell \in \N$ we have
$d_{m,n}\leq d_{m,p}+d_{p,n}$ for all $m,n,p \le \ell$. If any of $m$, $n$ or $p$ are equal to $-1$ this holds trivially, so that in particular the property holds
for the base case $\ell=-1$. Suppose that the property holds
    up to $\ell-1$ and let us prove it for $\ell$. Let $m$, $n$ and $p$ be non-negative and fix optimal transports $z^{m,p}$ for $d_{m,p}$ and $z^{p,n}$ for $d_{p,n}$. Define 
     $z_{i,j} = \sum^p_{k=0}w_{i,k,j} $ where
   \begin{equation*}
w_{i,k,j}=\left\{\begin{array}{cl}
z^{m,p}_{i,k}z^{p,n}_{k,j}/\pi_{k}^p&\mbox{if }\pi_{k}^p\neq 0,\\
0&\mbox{otherwise.}
\end{array}\right .
\end{equation*}
A straightforward verification shows that $ \sum^n_{j=0}w_{i,k,j}=z^{m,p}_{i,k}$ and $\sum^m_{i=0}w_{i,k,j}=z^{p,n}_{k,j}$, from which it also follows  that $z$ is a feasible transport from $\pi^m$ to $\pi^n$. Using the induction hypothesis we have
$c_{i-1,j-1}=\min\{1,\rho\, d_{i-1,j-1}\}\leq
\min\{1,\rho\, d_{i-1,k-1}+\rho\, d_{k-1,j-1}\}\leq c_{i-1,k-1}+c_{k-1,j-1}$, and then
    \begin{align*}
    d_{m,n} 
    & \le \sum^m_{i=0}\sum^n_{j=0}\sum_{k=0}^pw_{i,k,j}\, c_{i-1,j-1}
    \\& \le\sum^m_{i=0} \sum^n_{j=0}\sum_{k=0}^pw_{i,k,j}\, (c_{i-1,k-1}+c_{k-1,j-1})
     \\& = \sum^m_{i=0}\sum_{k=0}^pz^{m,p}_{i,k}\, c_{i-1,k-1}+\sum_{k=0}^p\sum^n_{j=0}z^{p,n}_{k,j}\,c_{k-1,j-1}
     \\& = d_{m,p}+d_{p,n}
\end{align*} 
which establishes the triangle inequality up to $\ell$ completing the induction step.   
\end{proof}

\begin{lemma}\label{sim_opt}
    For all $0\leq m\leq n$ there is an optimal transport for $d_{m,n}$ such that $z_{i,i}=\min\{\pi_i^m,\pi_i^n\}$. 
\end{lemma}

\begin{proof}
If $  m=0$ this is trivial. Otherwise, let $z$ be an optimal transport for $d_{m,n}$. If $z_{i,i}< \min\{\pi^n_i, \pi^m_i\}$ we must have $z_{i,k}>0$ for some $k \neq i$ and $z_{j,i}>0$ for some $j \neq i$. Decreasing $z_{i,k}$ and $ z_{j,i}$ by $\epsilon$ while increasing $z_{i,i}$ and $z_{j,k}$ by the
same amount, the modified transport is still feasible and the cost is reduced by
  \[(c_{i-1,i-1}+c_{j-1,k-1}-c_{j-1,i-1}-c_{i-1,k-1}) \,\epsilon \le 0\]
  so it remains optimal. Thus we can increase each $z_{i,i}$ up to $\min\{\pi^n_i, \pi^m_i\}$.
\end{proof}
Consider $0\leq m\leq n$ and recall that the primal optimal transport problem is the linear program
\begin{align*}
    d_{m,n} &= \min_{z} \sum^m_{i=0}\sum^n_{j=0}  z_{i,j}\, c_{i-1,j-1}\\
    & \quad \mbox{ s.t. } \mbox{$\sum^m_{i=0} z_{ij} = \pi^n_j, \quad   \sum^n_{j=0} z_{i,j} = \pi^m_i,  \quad z_{i,j} \ge 0.$}
\end{align*} 
Because $\pi^m_i=0$ for $i > m$, the dual problem can be written as 
\begin{align*}
  d_{m,n} &=  \max_{u} \sum^n_{i=0}u_i( \pi^m_i-\pi^n_i)\\
    & \quad \text{s.t.} \quad |u_i-u_j| \le c_{i-1,j-1}
 \end{align*} 
and each pair of primal-dual optimal solutions $z, u$ satisfy complementary slackness 
\[z_{i,j}\, c_{i-1,j-1}=z_{i,j}(u_i-u_j).\]

\begin{theorem}\label{thm:tightLip}
    For every $\rho>0$, $\kappa>0$, and triangular sequence $(\pi^n)_{n\in\NN}$ of averaging factors, there is  a $\rho$-Lipschitz map $T:X\to X$ on a normed space $(X,\|\cdot\|)$ and a corresponding Mann sequence $(x^n)_{n\in\NN}$ that satisfies \eqref{eq:bkappa} and such that $\|x^m-x^n\|=\kappa\,d_{m,n}$ and $\|x^n-Tx^n\|=\kappa\,R_n$ for all $m,n\in\NN$.
\end{theorem}

\begin{proof}
Again, by rescaling the norm it suffices to consider the case $\kappa=1$. For each $m\leq n$ take $z^{m,n}$ and $u^{m,n}$ primal and dual optimal solutions for the  optimal transport distances $d_{m,n}$. Setting $u^{m,n}_i \!= \min_{0 \le k \le n} u^{m,n}_k\!+c_{k-1,i-1} $ for $i>n$, the triangular inequality for the $c_{m,n}$'s implies
   \[ |u^{m,n}_i-u^{m,n}_j| \le c_{i-1,j-1} \mbox{ for all }\, i,j \in \mathbb{N},\]
which is a special case of the McShane-Whitney extension of Lipschitz functions. 
Since $c_{i-1,j-1}\leq 1$, it follows that all the $u^{m,n}_i$ differ at most by $1$ and, since the objective
function is invariant by translation, we may further assume that
$u^{m,n}_i\in [0,1]$ for all $i \in \mathbb{N}$. 

Let $Q$ be the set of all pairs $(m,n)$ of integers with $-1 \le m \le n$, and consider the unit cube $ C= [0,1]^Q$ in the space $(\ell^{\infty}(Q), \|\cdot \|_{\infty}) $. For every fixed $k\in \mathbb{N}$, define $y^k \in C$ as
   \begin{equation*}
y^{k}_{m,n}=\left\{\begin{array}{cl}
c_{k-1,n}&\mbox{if } \,-1 =m \le n\\
u^{m,n}_k&\mbox{if }\,\hspace{2.3ex} 0 \le m \le n
\end{array}\right .
\end{equation*}
and a corresponding sequence $x^k \in C$ given by
\begin{equation}\label{eq:MRLip}
x^k =\sum^k_{i=0}\pi^k_i y^i.
\end{equation}

\noindent\underline{{\sc Claim 1:} $\|y^{n+1}-y^{m+1}\|_\infty=c_{m,n}$ for all $-1 \le m \le n$}. \\[0.5ex]
The triangle inequality for the $c_{m,n}$'s and the dual feasibility of $u^{m'\!,n'}$ imply  respectively
$$
\left\{
\begin{array}{ll}
    |y^{n+1}_{-1,n'}-y^{m+1}_{-1,n'}|=|c_{n,n'}-c_{m,n'}| \le c_{m,n} & \mbox{ if } -1=m' \le n'\\[1ex]
    |y^{n+1}_{m'\!,n'}-y^{m+1}_{m'\!,n'}|=|u_{n+1}^{m'\!,n'}-u_{m+1}^{m'\!,n'}| \le c_{m,n} & \mbox{ if } \hspace{2.3ex}0\le m' \le n'
\end{array}\right.
$$
which combined yield
$\|y^{n+1}-y^{m+1}\|_\infty\leq c_{m,n}$.
By considering the case $m'=-1$ and $n'=n$ above we get
$|y^{n+1}_{-1,n}-y^{m+1}_{-1,n}|=|c_{n,n}-c_{m,n}|=c_{m,n}$ so that in fact $\|y^{n+1}-y^{m+1}\|_\infty= c_{m,n}$.

\noindent\underline{{\sc Claim 2:} $\|x^n-x^m\|_\infty=d_{m,n}$ for $0 \le m\le n$}. \\[0.5ex]
Using the optimal transport $z^{m,n}$ we get
\begin{align*}
    \|x^m-x^n\|_\infty&=\mbox{$\left\|\sum^m_{i=0}\pi^m_i y^i-\sum^n_{j=0}\pi^n_j y^j\right\|_\infty$}\\
    &=\mbox{$\left\|\sum^m_{i=0}\sum^n_{j=0}z^{m,n}_{ij}(y^i-y^j)\right\|_\infty$}\\
&\le \mbox{$\sum^m_{i=0}\sum^n_{j=0}z^{m,n}_{ij} c_{i-1, j-1}=d_{m,n}$},
\end{align*}
while considering the $(m,n)$-th coordinate and using complementary slackness we obtain 
\begin{align*}
    \|x^m-x^n\|_\infty&\geq | x^m_{m,n}-x^n_{m,n}|\\&= \mbox{$\left|\sum^m_{i=0}\sum^n_{j=0}z^{m,n}_{ij} (y^i_{m,n}-y^j_{m,n})\right|$}
    \\&= \mbox{$\left|\sum^m_{i=0}\sum^n_{j=0}z^{m,n}_{ij} (u_i^{m,n}-u_j^{m,n})\right|$}
     \\&= \mbox{$\sum^m_{i=0}\sum^n_{j=0}z^{m,n}_{ij} c_{i-1,j-1}=d_{m,n}$}.
\end{align*}

\noindent\underline{\sc Conclusion}:
Form the previous claims we have $\|y^{n+1}-y^{m+1}\|_\infty=c_{m,n}\leq\rho\,d_{m,n}=\rho\|x^m-x^n\|_\infty$. 
Hence, defining $T: S \rightarrow C$ on the set $S=\{x^k: k \in \mathbb{N} \} \subseteq C$ by $Tx^k=y^{k+1}$, it follows that $T$ is $\rho$-Lipschitz. Since $C\subseteq \ell^{\infty}(Q)$ is hyperconvex, applying Theorem 4 of Section \S{}2 in \cite[Aronszajn-Panitchpakdi]{ap1956}, $T$ can be extended to a $\rho$-Lipschitz map $T: \ell^{\infty}(Q) \rightarrow C$. 
From \eqref{eq:MRLip} we have that $(x^n)_{n\in\NN}$ is precisely a Mann sequence 
for this map, such that
$\|x^0-Tx^n\|_\infty=\|y^0-y^{n+1}\|_\infty=c_{-1,n}=1$  and $\|Tx^m-Tx^n\|_\infty=\|y^{m+1}-y^{n+1}\|_\infty=c_{m,n}\leq 1$ so that \eqref{eq:bkappa} is satisfied. Moreover, by {\sc Claim 2} we have
$\|x^n-x^m\|_\infty=d_{m,n}$ for all $m,n\in\NN$, so that  it remains to show that $\|x^n-Tx^n\|_{\infty}=R_n$. The upper bound follows again by using a triangle inequality 
\begin{align*}
    \|x^n-Tx^n\|_{\infty}=\mbox{$\left\|\sum^n_{i=0}\pi^n_i(y^i-y^{n+1})\right\|_{\infty}\le \sum^n_{i=0}\pi^n_ic_{i-1,n} = R_n$}
\end{align*}
while the reverse inequality follows by considering  the $(-1,n)$-th coordinate
\begin{align*}
    \|x^n-Tx^n\|_{\infty} &\ge \mbox{$\left|\sum^n_{i=0}\pi^n_i y^i_{-1,n}-y^{n+1}_{-1,n}\right|
     = \left|\sum^n_{i=0}\pi^n_i c_{i-1,n}-c_{n,n}\right|=R_n.$}
\end{align*}
\end{proof}

\begin{corollary}\label{Cor_6.6Lip}
The error bounds for Halpern's iterates in Proposition \ref{Prop_2.1} are tight.
\end{corollary}

\begin{proof}
As   we may assume that $\kappa=1$. Proposition \ref{Prop_2.1} established 
the bounds $\|x^n-x^{n-1}\|\leq d_n$ and  $\|x^n-Tx^n\|\le 1-\beta_n+\beta_n\,c_n$ where $c_n=\min\{1,\rho \,d_n\}$ and $d_n=|\beta_{n-1}-\beta_n|+\min\{\beta_{n-1},\beta_n\}c_{n-1}$.

We claim that these bounds coincide with those  obtained from optimal transport, so that
their tightness follows from Theorem \ref{thm:tightLip}. Now,  Halpern is the special case of Mann's iterates where $\pi_0^n=1-\beta_n$ and $\pi_n^n=\beta_n$, which yields the optimal transport bound
$\|x^n-Tx^{n}\|\leq R_n=1-\beta_n+\beta_{n}\, c_{n-1,n}$.
Moreover, due to Lemma \ref{sim_opt}, for each $m\leq n$ there is a unique simple transport from $\pi^m$ to $\pi^n$, namely

\begin{tabular}{cl}
{\sc Case $\beta_n\geq\beta_m$ :}& $z_{00}=1-\beta_n$;
$z_{0n}=\beta_n-\beta_m$; $z_{mn}=\beta_m$\\
{\sc Case $\beta_n\leq\beta_m$ :}& $z_{00}=1-\beta_m$;
$z_{m0}=\beta_m-\beta_n$; $z_{mn}=\beta_n$
\end{tabular}

\noindent which combined give
$$d_{m,n}=|\beta_m-\beta_n|+\min\{\beta_m,\beta_n\}\,c_{m-1,n-1}.$$
In particular, for $m=n-1$ it follows that $c_{n-1,n}$ and $d_{n-1,n}$ satisfy the same recursion as $c_n$ and $d_n$ in \eqref{eq:rec}, so that the optimal transport bounds coincide with those in Proposition \ref{Prop_2.1} as claimed.
\end{proof}

\section{Comparing bounds in normed and Hilbert spaces}\label{AppendixB}

\begin{lemma}\label{Le:Mon0} The sequence $(\rho_n^{n})_{n\in\NN}$ decreases towards $e^{-2}$.
\end{lemma}

\begin{proof}
The recursion $\rho_n=\frac{1}{2}(1+\rho_{n-1}^2)$ with $\rho_0=\frac{1}{2}$
 increases to 1 with  
$\epsilon_n\triangleq n(1-\rho_n)\to 2$, so that a standard
result gives $\rho_n^{n}=(1-\frac{\epsilon_n}{n})^{n}\to e^{-2}$.
Let us now show that $n\ln(\rho_n)$ decreases, which amounts to
$n\ln(\frac{\rho_{n+1}}{\rho_n})+\ln(\rho_{n+1})\leq 0$.
From $\rho_{n+1}\!=\frac{1}{2}(1+\rho_{n}^2)$ we have $\rho_{n+1}\!-\rho_n=\frac{1}{2}(1-\rho_{n})^2$ so that $\frac{\rho_{n+1}}{\rho_n}\!=1+\frac{(1-\rho_{n})^2}{2\rho_n}$ and therefore $\ln(\frac{\rho_{n+1}}{\rho_n})\leq \frac{(1-\rho_{n})^2}{2\rho_n}$.
Similarly $\ln(\rho_{n+1})\leq\rho_{n+1}-1= \frac{\rho_n^2-1}{2}$, and therefore
$$\mbox{$n\ln(\frac{\rho_{n+1}}{\rho_n})+\ln(\rho_{n+1})\leq n\frac{(1-\rho_{n})^2}{2\rho_n}+\frac{\rho_n^2-1}{2}=\frac{(1-\rho_n)}{2\rho_n}\big(n(1-\rho_n)-\rho_n(1+\rho_n)\big)$}.$$
We claim that $n(1-\rho_n)-\rho_n(1+\rho_n)\leq 0$. Indeed, the quadratic $n(1-x)-x(1+x)$ is negative for  $x$ larger than the positive root $x_n=\frac{1}{2}(\sqrt{1+6n+n^2}-(n+1))$, so the result follows  since $\rho_n\geq x_n$ which is proved inductively from $\rho_{n+1}=\frac{1}{2}(1+\rho_n^2)$.
\end{proof}

\begin{lemma}\label{Le:Mon1}
For each $\rho\in[0,1)$ the quotients $Q_n(\rho)$ increase with $n$.
\end{lemma}

\begin{proof} 
Let us  show that for all $n\geq 1$ we have $\Delta_n(\rho)\leq 1$ 
where $$\Delta_n(\rho)=\frac{Q_{n-1}(\rho)}{Q_n(\rho)}=\frac{\rho\,(1-\rho^{n})}{(1-\rho^{n+1})}\frac{R_{n-1}(\rho)}{R_n(\rho)}.$$ 
For $n>n_0(\rho)$ we have $R_n^*(\rho)=\Vopt(R_{n-1}^*(\rho))=\rho\, R_{n-1}^*(\rho)$ so that $\Delta_n(\rho)=\frac{(1-\rho^{n})}{(1-\rho^{n+1})}\leq 1$. For $n\leq n_0(\rho)$ 
we have $R_n^*(\rho)=1-z_n/\rho$ so that 
$$\Delta_n(\rho)=\frac{\rho\,(1-\rho^{n})}{(1-\rho^{n+1})}\frac{(\rho-z_{n-1})}{(\rho-z_n)}$$ 
and then, after some simple algebra, it follows that $\Delta_n(\rho)\le 1$ if and only if 
\begin{equation*}
(1-\rho)(\rho-z_n)\geq \rho\,(1-\rho^{n}) (z_n-z_{n-1}).
\end{equation*}
Divide by $(1-\rho)$ and let $H_n(\rho)\triangleq (\rho-z_n)-\frac{1-\rho^{n}}{1-\rho}\,\rho\, (z_n-z_{n-1})$. We must show that
$H_n(\rho)\geq 0$ for all $n\leq n_0(\rho)$ or, equivalently, for all   $\rho\in [\rho_n,1)$.
Noting that $\rho\mapsto H_n(\rho)$ is concave, it suffices to check that this  holds for $\rho=1$ and $\rho=\rho_n$. \\[1ex]
\underline{{\sc Case 1}: $H_n(1)\geq 0$}. Since $H_n(1)=(1-z_n)-n(z_n-z_{n-1})$, this inequality amounts to $(n+1)(1-z_n)\geq n(1-z_{n-1})$ which follows by a simple induction.\\[1ex]
\underline{{\sc Case 2}: $H_n(\rho_n)\geq 0$}. Using the identities $\rho_n=\frac{1}{2}(1+z_n)$ 
and $z_n-z_{n-1}=\frac{1}{4}(1-z_{n-1})^2$,
the inequality $H_n(\rho_n)\geq 0$ turns out to be equivalent to 
$$\mbox{$\frac{2}{1+z_n}\big(\frac{1-z_n}{1-z_{n-1}}\big)^2$}\geq (1-\rho_n^{n}).$$
From $z_n=\frac{1}{4}(1+z_{n-1})^2$ we get $z_{n-1}=2\sqrt{z_n}-1$ and the expression on the left is $W(z_n)$ with
$W(x)=\frac{1}{2}+\frac{\sqrt{x}}{1+x}$. Since this map is increasing it follows that $W(z_n)\geq W(z_1)=W(\frac{1}{4})=\frac{9}{10}$, while the
right hand side $(1-\rho_n^{n})$ increases to $1-e^{-2}$ which is smaller than
$\frac{9}{10}$.
\end{proof}

\begin{lemma}\label{Le:Mon2}
The function $\rho\mapsto Q_\infty(\rho)$ is continuous and increasing for $\rho\in [0,1]$.
\end{lemma}

\begin{proof}
As in the proof of Proposition \ref{Prop:Four}, for all $\rho\in I_n=(\rho_{n-1},\rho_n]$ we have that  $n_0(\rho)\equiv n$ is constant and 
$Q_\infty(\rho)=\frac{\rho-z_n}{(1-\rho)\rho^{n+1}}$. Moreover, from $z_{n+1}=\frac{1}{4}(1+z_n)^2$ it follows that the expressions $\frac{\rho-z_n}{(1-\rho)\rho^{n+1}}$ and $\frac{\rho-z_{n+1}}{(1-\rho)\rho^{n+2}}$
  coincide at  the interface $\rho_n$ between $I_n$ and $I_{n+1}$, and therefore
 $\rho\mapsto Q_\infty(\rho)$ is continuous over the full interval $\rho\in[0,1)$. 
 In order to establish the monotonicity we show that $Q_\infty'(\rho)\geq 0$ for $\rho\in I_n$. Denoting $w_n=\frac{n+(n+2)z_n}{2(n+1)}$ we have 
 $$Q_\infty'(\rho)=\frac{(n+1)}{(1-\rho)^2\rho^{n+2}}\big((\rho-w_n)^2+z_n-w_n^2\big)$$ 
 so it suffices to check that $z_n\geq w_n^2$. Since 
 $z_n-w_n^2=\frac{(n+2)^2}{4(n+1)^2}(1-z_n)(z_n-(\frac{n}{n+2})^2)$ this follows
 using a simple inductive argument to show that $(\frac{n}{n+2})^2\leq z_n< 1$.
 \end{proof}

\section{Proofs of the results in Section \S{}\ref{sec3}}

\subsection{Proof of Proposition~\ref{Prop_2.1Bis}.} \label{ApC1}
\begin{proof}
The bound $\|x^n\!-x^*\|\leq\delta_0\mu_n$ follows inductively from $\|x^0\!-x^*\|=\delta_0\mu_0$ and 
\begin{align*}
  \|x^n\!-x^*\|&\le (1\!-\!\beta_{n}) \|x^0\!-x^*\|+ \beta_{n}\|Tx^{n-1}-x^*\|
    \\
&\le  (1\!-\!\beta_{n})\delta_0+\rho \,\beta_{n} \|x^{n-1}\!-x^*\|\\
&\le  \delta_0(1\!-\!\beta_{n}+\rho \,\beta_{n}\mu_{n-1})=\delta_0\mu_n.
\end{align*}
Then, a simple triangle inequality yields
$$
 \|x^0\!-Tx^n\|\le \|x^0\!-x^*\|+\|x^*\!-Tx^n\|\leq \delta_0+\rho\,\delta_0\mu_n=\delta_0\nu_n.
$$
On the other hand, $\|x^1\!-x^0\|=\beta_1\|x^0\!-Tx^0\|\leq\beta_1\delta_0\nu_0=\delta_0 \dbis_1$, and inductively
\begin{align*}
 \|x^n\!-x^{n-1}\| 
  &= \|(\beta_{n-1}\!-\!\beta_n)(x^0\!-Tx^{n-1})+\beta_{n-1}(Tx^{n-1}\!-Tx^{n-2})\|\\
 &\le (\beta_{n}\!-\!\beta_{n-1})\delta_0\nu_{n-1}+\rho\,\beta_{n-1}\|x^{n-1}\!-x^{n-2}\|
 \\
 &\le \delta_0\big((\beta_{n}\!-\!\beta_{n-1})\nu_{n-1}+\rho\,\beta_{n-1}\dbis_{n-1}\big)=\delta_0\dbis_n,
\end{align*}
from which it also follows that
\begin{align*}
    \|x^n\!-Tx^n\| &= \|(1\!-\!\beta_{n})(x^0\!-Tx^n)+\beta_{n}(Tx^{n-1}\!-Tx^{n})\|
    \\&\le (1\!-\!\beta_{n})\delta_0\nu_n+\rho\,\beta_{n}\,\delta_0\dbis_n = \delta_0\Rbis_n.
\end{align*}
Hence, using the recursive expressions for $\nu_n$ and $\dbis_n$ we obtain
\begin{align*}
    \Rbis_n&=(1\!-\!\beta_{n}) \nu_n +\rho\,\beta_{n}\,\dbis_n\\
    &=  (1\!-\!\beta_{n}) \big(1+\rho\, \mu_n\big) +\rho\,\beta_{n} \big((\beta_{n}\!-\!\beta_{n-1})\nu_{n-1}+\rho\,\beta_{n-1}\,\dbis_{n-1}\big)
        \\&= (1\!-\!\beta_{n})(1+\rho\, \mu_n) +\rho\,\beta_{n} \big((\beta_{n}\!-\!1)\nu_{n-1}+\Rbis_{n-1}\big)
        \\&= (1-\beta_n)(1+\rho-2\rho\beta_n)+\rho\,\beta_n\,\Rbis_{n-1}
\end{align*}
which after simplification gives precisely \eqref{eq:Rec1}.
\end{proof}

\subsection{Convergence of \texorpdfstring{\flathalpern{}}{} iterates}\label{ApC2}
Let us now assume that the space $(X,\|\cdot\|)$ is complete, and let $T\in\Lip(\rho)$ with some fixed point $x^*\!=Tx^*$ so that condition \eqref{eq:fixed_point} in section  \S\ref{sec3.1} is satisfied.
Consider the sequence $(x^n)_{n\in\NN}$ produced by \flathalpern{}. We claim 
that these iterates converge, except perhaps when $\rho=1$.

 For $\rho\geq\sqrt{2}+1$ this is trivial (and rather uninformative) since the sequence $x^n\equiv x^0$ is constant.
A more interesting but still easy case is when $\rho<1$, where $x^n$ converges to $x^*\!$. Indeed, from $\|x^n\!-x^*\|\leq \|x^n\!-Tx^n\|+\|Tx^n\!-Tx^*\|$ it follows that 
$$\|x^n\!-x^*\|\leq \|x^n\!-Tx^n\|/(1\!-\!\rho)\leq\delta_0\Rbis_n/(1\!-\!\rho)\to 0.$$

Let us then analyze the more subtle case is when $\rho\in (1,\sqrt{2}+1)$. Recall that
the limit map $T^\flat_{\!\rho}x\triangleq (1\!-\!\betabis_{\!\rho})x^0+\betabis_{\!\rho}\,Tx$ is $L$-Lipschitz with $L=\rho\,\betabis_{\rho}<1$,
 so it has a unique fixed point $x^\flat_{\rho}=T^\flat_{\!\rho}x^\flat_{\rho}$. 

\begin{theorem}
     For any $T\in\Lip(\rho)$ with $\rho\in (1,\sqrt{2}+1)$, the sequence $(x^n)_{n\in\NN}$ produced by \flathalpern{} satisfies $\|x^n\!-\xflat\|\leq \|x^0\!-\xflat\|(n+1)L^n\to 0.$
\end{theorem}

\begin{proof} Let us recall that $\betabis_n\uparrow\betabis_{\!\rho}$. For $\rho\in (1,\sqrt{2}+1)$ we have 
\begin{align*}
    \betabis_n&=\beta_\rho(\Rbis_{n-1})=(1/\rho+3-\Rbis_{n-1})/4\\
    \Rbis_n&=\Voptbis(\Rbis_{n-1})=(1+\rho)-2\rho (\betabis_n)^2
\end{align*}
 which combined imply that 
\begin{align*}
\betabis_n&=(1/\rho+2-\rho+2\rho (\betabis_{n-1})^2)/4\\
\betabis_{\!\rho}&=(1/\rho+2-\rho+2\rho (\betabis_{\!\rho})^2)/4.
\end{align*}
Subtracting the latter equalities we get  
$$0\leq \betabis_{\!\rho}-\betabis_n=\frac{\rho}{2} (\betabis_{\!\rho}+ \betabis_{n-1})(\betabis_{\!\rho}- \betabis_{n-1})\leq L(\betabis_{\!\rho}- \betabis_{n-1})$$
and therefore 
\begin{equation}\label{eq:bbb}
0\leq \betabis_{\!\rho}-\betabis_n\leq L^n(\betabis_{\!\rho}-\betabis_0)=L^n\,\betabis_{\!\rho}.
\end{equation}

We now proceed to establish the following recursive bound
\begin{align*}
    \|x^n\!-x^\flat_{\rho}\|&=\|(1-\betabis_n)x^0+\betabis_nTx^{n-1}-(1-\betabis_{\!\rho})x^0-\betabis_{\!\rho}T\xflat\|\\
    &=\|(\betabis_{\!\rho}-\betabis_n)x^0+\betabis_nTx^{n-1}-\betabis_{\!\rho}T\xflat\|\\
    &=\|(\betabis_{\!\rho}-\betabis_n)(x^0-T\xflat)+\betabis_n(Tx^{n-1}-T\xflat)\|\\
    &\leq (\betabis_{\!\rho}-\betabis_n)\|x^0-T\xflat\|+\betabis_n\rho\|x^{n-1}-\xflat\|\\
    &= \frac{1}{\betabis_{\!\rho}}(\betabis_{\!\rho}-\betabis_n)\|x^0-\xflat\|+\betabis_n\rho\|x^{n-1}-\xflat\|
\end{align*}
where the last equality results from the equation $\xflat=T^\flat_{\!\rho}\xflat$. Using \eqref{eq:bbb} and $\betabis_n\rho\leq L$ we get
$$ \|x^n\!-x^\flat_{\rho}\|\leq
L^n\|x^0-\xflat\|+L\|x^{n-1}-\xflat\|$$
and then the conclusion follows by a simple induction.
\end{proof}

\begin{remark}
From the previous theorem we can also derive a bound in terms of $\delta_0\geq \|x^0\!-x^*\|$, since
$\|x^0\!-\xflat\|\leq (\sqrt{2}+\!1)\delta_0/\rho\leq (\sqrt{2}+\!1)\delta_0$.

\begin{proof}
For $0\leq\beta<1/\rho$ the map $T_\beta\hspace{0.2ex} x\triangleq (1\!-\!\beta)x^0+\beta \,Tx$ is a contraction. Its unique fixed point $x_\beta=T_{\!\beta}x_\beta$ satisfies
\begin{align*}
    \|x_\beta-x^*\|&=\|T_\beta x_\beta-Tx^*\|\\
    &=\|(1-\beta)(x^0-Tx^*)+\beta\,(Tx_\beta-Tx^*)\|\\
    &\leq(1-\beta)\|x^0-x^*\|+\beta\rho\|x_\beta-x^*\|
\end{align*}
and therefore\footnote{This estimate is tight: for $Tx=\rho x$ with $\rho>1$ the unique fixed points are $x^*=0$ and $x_\beta=\frac{1-\beta}{1-\beta\rho}x^0$.}
$\|x_\beta-x^*\|\leq\frac{1-\beta}{1-\beta\rho}\|x^0-x^*\|$.
In particular, for $\beta=\betabis_{\!\rho}=\frac{\sqrt{2}+1-\rho}{\sqrt{2}\,\rho}$ we have $x_\beta=x^\flat_\rho$, and a direct substitution yields
$$\|\xflat-x^*\|\leq(\sqrt{2}+\!1)\|x^0\!-x^*\|/\rho.$$

\end{proof}
\end{remark}

\subsection{Proof of Theorem~\ref{Thm:affine}.}
\label{ApC3}

\begin{proof} Let us reformulate  $L_n^*$ using the alternative variables $z_i=\B^n_{i+1}- B^n_{i}$ for $i=0,\ldots,n$. By telescoping it follows  $\sum_{i=0}^{n}z_i=1$ while  $\beta_i\in[0,1]$ is equivalent $z_i\geq 0$. Thus, the vector  $z=(z_0,\ldots,z_{n})$ belongs to the unit simplex $\Delta^{n+1}$ and
therefore
\begin{align*}
    L_n^*&=\min_{z\in\Delta^{n+1}} |z_{n}|+\rho^{n+1}|z_{0}|+\mbox{$\sum_{i=1}^{n} \rho^{n+1-i}|z_{i}-z_{i-1}|$}.
\end{align*}

\noindent\underline{\sc Case $\rho<1$}:
We claim that in an optimal solution the $z_i$'s must be decreasing. Indeed, if  we had $\epsilon=z_{k+1}-z_{k}>0$ for some $k\in\{1,\ldots,n\}$, we could find a strictly better solution $\tilde z\in\Delta^{n+1}$ by setting $\tilde z_k=(z_{k}+\epsilon)/(1+ \epsilon)$ and 
$\tilde z_i=z_{i}/(1+ \epsilon)$ for $i \neq k$. Restricting the minimization to the polytope $\{ z_0\ge z_1 \ge \dots \ge z_{n}  \,:\, z \in \Delta^{n+1}\}$
we can ignore the absolute values in the objective function and $L_n^*$ becomes a linear program whose minimum is attained at an extreme point.
The extreme points are of the form 
$z_0=\ldots=z_k=\frac{1}{k+1}$ and $z_{k+1}=\ldots=z_n=0$ for $k\in\{0,\ldots,n\}$ with
 value $\rho^{n-k}(1+\rho^{k+1})/(k+1)$. This latter expression
is minimized  precisely at $k=\min\{n,n_0\}$ which gives \eqref{eq:optRn}. In order to recover the optimal $\beta_i$'s we notice that by telescoping we have  $$\mbox{$\B_i^n$}=z_0+\ldots+z_{i-1}=\left\{
\begin{array}{cl}
i/(k+1)&\mbox{if }i\leq k\\
1&\mbox{if }i> k
\end{array}\right.
$$
from which obtain that 
$\beta_i=i/(i+1)$ for $i\leq n_0$ and $\beta_i=1$ for $i>n_0$.

\noindent\underline{\sc Case $\rho>1$}: Similarly to the previous case, the optimal $z_i$'s are now increasing and the minimization can be restricted to the polytope $\{ z_0\le z_1 \le \dots \le z_{n}  \,:\, z \in \Delta^{n+1}\}$.
The extreme points are $z_0=\ldots=z_{n-k-1}=0$ 
and $z_{n-k}=\ldots=z_n=\frac{1}{k+1}$  with $k\in\{0,\ldots,n\}$, which attain the value $(1+\rho^{k+1})/(k+1)$. This later expression decreases for $k\leq n_0$ and increases afterwards so that the minimum is attained for $k=\min\{n,n_0\}$ and yields \eqref{eq:optRn_exp}.
As before, we recover the optimal $\beta_i$'s by telescoping. When $n\leq n_0$ we get $\beta_i=i/(i+1)$, whereas 
for $n>n_0$ the optimal algorithm makes 
    $n-n_0$ null steps staying at $x^0$ and then performs $n_0$ further steps.
An alternative way to express this is that the optimal Halpern runs $n_0$ steps with $\beta_i=\frac{i}{i+1}$ for $i=1,\ldots,n_0$ and then takes $x^n\equiv x^{n-1}$ for all $n>n_0$.
\end{proof}

\end{document}